\documentclass[10pt,a4paper,twoside,reqno]{amsart}

\usepackage{amssymb}
\usepackage{mathtools}
\usepackage[shortlabels]{enumitem}
\usepackage[utf8]{inputenc}
\usepackage{color}
\usepackage{mathabx}
\usepackage{tikz}
\usetikzlibrary{positioning,decorations.pathreplacing,decorations.markings,arrows.meta,calc,fit,quotes,cd,math,arrows,backgrounds,shapes.geometric}
\usepackage[all,2cell]{xy}\UseAllTwocells\SilentMatrices
\usepackage[a4paper,top=3cm,bottom=3cm,inner=2.5cm,outer=2.5cm]{geometry}
\definecolor{darkgreen}{rgb}{0,0.45,0}
\usepackage[colorlinks,citecolor=darkgreen,urlcolor=gray,final,hyperindex,linktoc=page,hyperfootnotes=true]{hyperref}
\usepackage[capitalize]{cleveref}
\usepackage{mathrsfs}
\usepackage{comment}

\crefname{equation}{}{}
\crefname{thm}{Theorem}{Theorems}
\crefname{defi}{Definition}{Definitions}
\crefname{rmk}{Remark}{Remarks}
\crefname{prop}{Proposition}{Propositions}
\crefname{ex}{Example}{Examples}
\crefname{cor}{Corollary}{Corollaries}
\crefformat{footnote}{#2\footnotemark[#1]#3}

\theoremstyle{plain}
\newtheorem{thm}{Theorem}[section]
\newtheorem{cor}[thm]{Corollary}
\newtheorem{lem}[thm]{Lemma}
\newtheorem{prop}[thm]{Proposition}

\theoremstyle{remark}
\newtheorem{rmk}[thm]{Remark}
\newtheorem{ex}[thm]{Example}

\theoremstyle{definition}
\newtheorem{defi}[thm]{Definition}

\definecolor{mypurple}{rgb}{0.5, 0.0, 0.5}
\newcommand{\chris}{\color{mypurple}Christina: }
\newcommand{\vas}{\color{blue}Vassilis: }
\definecolor{teocolor}{rgb}{0.0, 1.0, 0.0}
\newcommand{\teo}[1]{\textcolor{teocolor}{Theofilos: #1}}

\tikzstyle{start}=[to path={(\tikztostart.#1) -- (\tikztotarget)}]
\tikzstyle{end}=[to path={(\tikztostart) -- (\tikztotarget.#1)}]

\newcommand{\ca}{\mathcal}
\newcommand{\dc}{\mathbb}
\newcommand{\nc}{\mathsf}
\newcommand{\Gr}{\mathfrak}

\newcommand{\B}{\mathbf}
\newcommand{\wc}{\widecheck}
\newcommand{\wh}{\widehat}

\newcommand{\mi}{\textrm{-}}
\newcommand{\ot}{\otimes}
\newcommand{\ob}{\ensuremath{\mathrm{ob}}}
\newcommand{\Set}{\nc{Set}}
\newcommand{\Cat}{\nc{Cat}}
\newcommand{\VCat}{\ca{V}\textrm{-}\nc{Cat}}
\newcommand{\VCocat}{\ca{V}\textrm{-}\nc{Cocat}}
\newcommand{\VGrph}{\ca{V}\textrm{-}\nc{Grph}}
\newcommand{\Coalg}{\nc{Coalg}}
\newcommand{\Alg}{\nc{Alg}}

\newcommand{\Mon}{\nc{Mon}}
\newcommand{\Comon}{\nc{Comon}}
\newcommand{\End}{\nc{End}}
\newcommand{\Mnd}{\nc{Mnd}}
\newcommand{\Cmd}{\nc{Cmd}}
\newcommand{\Mod}{\nc{Mod}}
\newcommand{\Comod}{\nc{Comod}}

\newcommand{\op}{\mathrm{op}}

\newcommand{\id}{\mathrm{id}}

\tikzset{tick/.style={postaction={decorate,decoration={markings,mark=at position 0.5 with {\draw[-] (0,.5ex) -- (0,-.5ex);}}}}}
\tikzset{bul/.style={postaction={decoration={markings,mark=at position 0.5 with {\node{$\sbul$};}},decorate}}}

\tikzset{tick/.style={postaction={decorate,decoration={markings,mark=at position 0.5 with {\draw[-] (0,.4ex) -- (0,-.4ex);}}}}}
\newcommand{\tickar}{
  \begin{tikzcd}[baseline=-0.5ex,cramped,sep=small,ampersand
    replacement=\&]{}\ar[r,tick]\&{}
\end{tikzcd}}
\newcommand{\bular}{
  \begin{tikzcd}[baseline=-0.5ex,cramped,sep=small,ampersand
    replacement=\&]{}\ar[r,bul]\&{}
\end{tikzcd}}

\newcommand{\VMMat}{\ca{V}\mi\dc{M}\nc{at}}
\newcommand{\sbul}{\scriptstyle\bullet}
\newcommand{\tick}{\object@{|}}

\newcommand{\Span}{\dc{S}\nc{pan}}

\newcommand{\matr}[3]{\SelectTips{eu}{10}\xymatrix@C=.2in{#1\colon #2\ar[r]|-{\object@{|}} & #3}}
\newcommand{\Two}{\scriptstyle\Downarrow}

\newcommand{\dcD}[2]{\dc{D}[{#1},{#2}]}

\newcommand{\bultwocell}[9]{
  \begin{tikzcd}[ampersand replacement=\&,sep=.3in]
    #1\ar[r,bul,"{#2}"]\ar[d,"{#8}"']\ar[dr,phantom,"\Two{#9}"] \& #3\ar[d,"{#4}"] \\
    #7\ar[r,bul,"{#6}"'] \& #5
\end{tikzcd}}
\newcommand{\globtwocell}[7]{
  \begin{tikzcd}[ampersand replacement=\&]
    #1\ar[r,bul,"{#2}"]\ar[d,equal]\ar[dr,phantom,"\Two{#7}"] \& #3\ar[d,equal] \\
    #6\ar[r,bul,"{#5}"'] \& #4
\end{tikzcd}}

\begin{document}
\title{On categories of monads and comonads in double categories}

\author{Vasileios Aravantinos-Sotiropoulos}
\address{School of Applied Mathematical and Physical Sciences, National Technical University of Athens, Greece}
\email{v\_aravantinos@mail.ntua.gr}

\author{Theofilos Tsantilas}
\address{School of Applied Mathematical and Physical Sciences, National Technical University of Athens, Greece}
\email{teo.tsantilas@pm.me}

\author{Christina Vasilakopoulou}
\address{School of Applied Mathematical and Physical Sciences, National Technical University of Athens, Greece}
\email{cvasilak@math.ntua.gr}

\begin{abstract}
  As is well known in the literature, the category $\Mon(\ca{V})$ of monoids in a monoidal category $\ca{V}$ satisfies various fundamental categorical properties, at least when the monoidal base $\ca{V}$ is correspondingly well-behaved. In particular, $\Mon(\ca{V})$ is monadic over $\ca{V}$ as soon as free monoids exist, while if $\ca{V}$ is cocomplete or locally presentable and its tensor $\otimes$ is sufficiently compatible with the appropriate colimits, then $\Mon(\ca{V})$ inherits the analogous property. 
  
In the present work, we extend such results to the context of double categories. More precisely, we identify conditions on a double category $\dc{D}$ under which one can show that the category $\Mnd(\dc{D})$ of monads in $\dc{D}$ is monadic over the category of endomorphisms $\End(\dc{D})$, is cocomplete or even locally presentable. We also tackle the issue of local presentability in the dual case $\Cmd(\dc{D})$ of comonads. In these results, our assumptions on the double category $\dc{D}$ revolve around notions of colimit, in particular those of \emph{parallel} and \emph{stable local} colimits, as well as a notion of \emph{local presentability} of a double category which has been introduced in previous work.
\end{abstract}

\maketitle

\setcounter{tocdepth}{1}
\tableofcontents

\section{Introduction}

Double categories were introduced in the '60s by Charles Ehresmann \cite{Ehresmanndouble}, and have since been developed by multiple authors in various seminal works like \cite{DawPar,Limitsindoublecats,Adjointfordoublecats,ModelDC,Framedbicats,NiefieldSpan}.
 More recently, the subject has experienced somewhat of a renaissance, with numerous authors developing a wide spectrum of double-categorical subjects, spanning also various generalisations -- such as virtual double categories, enrichment in double categories and notions of double-categorical logic. A few indicative, but certainly not exhaustive, references are \cite{aleiferi,2Catmod,Clarke2022,Doublefibs,Nathanael1,DavidJazSophie}.

Concerning the current work, the concept of a monad, together with its dual of comonad, have of course been of central importance in ordinary category theory and subsequently in 2-dimensional category theory since the very beginning of the subject, see e.g. \cite{FormalTheoryMonadsI}. Although it is well known that a monad in a double category, as an object, reduces to that of a monad in its corresponding (horizontal) bicategory, the move to the double-categorical notion is not at all meaningless. In fact, the extra degree of freedom afforded to the 2-cells in a double category is precisely what is needed so that the associated notion of monad morphism captures what it ``morally should'' in many examples, contrary to the corresponding bicategorical notion. This will also be illustrated by the examples included in the first section below.

This manuscript can thus be positioned among works that explore fundamental properties of categories of monads (previously called monoids in certain references) in double categories, like \cite{Framedbicats,Monadsindoublecats,Doubleadjunctionsandfreemonads,SweedlerDouble,CommutingTensorProducts}.
More precisely, the purpose of this work is to identify conditions on a double category $\dc{D}$ that will imply certain desirable properties for the category $\Mnd(\dc{D})$ of monads therein, and their dual counterparts -- to the extent that they hold -- for the category of $\Cmd(\dc{D})$ of comonads. The particular properties we are interested in concern the existence of free monads, cocompleteness, monadicity over the category of endomorphisms and local presentability for $\Mnd(\dc{D})$, and dually the existence of free comonads, comonadicity and also local presentability for $\Cmd(\dc{D})$. The relevant conditions for the surrounding double category $\dc{D}$ are quite natural and divide in two different sets: the one resembles a more standard bicategorical setting -- namely a fibrant double category whose bicategory has sufficiently many local colimits -- and applies only to the case of monads, and the other set of conditions assumes local presentability on $\dc{D}$ itself, a concept introduced in the recent \cite{SweedlerDouble}.

More explicitly, \cref{prop:freemonadfunctorexistence1,cor:monadicity,cor:Mnd(D)cocomplete} establish existence of free monads, monadicity and cocompleteness of the category of monads under the first set of assumptions, \cref{prop:freemonadfunctorexistence2,cor:monadicity,thm:Mndlp} establish existence of free monads, monadicity and local presentability of the category of monads under the second set of assumptions, and \cref{prop:freecomonadexistence,cor:Cmd(D)comonadic,thm:Cmdlp} establish existence of free comonads, comonadicity and local presentability of the category of comonads -- bearing in mind that local presentability does not dualize. Very broadly speaking, the first trail concerning monads extends results from \cite{Varthrenr} to the double-categorical context, whereas the other trails use methods inspired by \cite{MonComonBimon,LocallyPresentable}. Moreover, as an application of the fact that the category of monads is locally presentable, we obtain the known fact that $\ca{V}$-$\mathsf{Cat}$ is locally presentable when $\ca{V}$ is, see \cite{KellyLack}.

In the process of proving these fundamental results, we make a choice, in the first \cref{sec:doublecats}, of illustrating key concepts such as parallel limits and colimits, stable local limits and colimits and local presentability in a number of basic examples of double categories without assuming prior familiarity with those, such as relations, spans, matrices and polynomials. In each of these cases, fibrancy -- the existence of companions and conjoints that give a coherent passage from the vertical/tight to the horizontal/loose structure -- allows us to employ results from the theory of (op)fibrations and relate structures from the fibres to the total categories and vice versa. Such techniques are expected to carry on such developments onto the categories of (one-sided) modules for monads, comodules for comonads as well as bimodules and bicomodules in future work.


\subsubsection*{Acknowledgements}
The authors would like to thank Nathanael Arkor, Florian De Leger, Ignacio L\'opez Franco and Ross Street for helpful discussions. All authors acknowledge that this work was implemented in the framework of H.F.R.I call ``3rd Call for H.F.R.I.’s Research Projects to
Support Faculty Members \& Researchers'' (H.F.R.I. Project Number: 23249). The second author was financially supported by the Special Account for Research Funding of the National Technical University of Athens.

\section{Double categories}\label{sec:doublecats}

In this section,
we start by recalling (fibrant) double categories together with certain results relating to fibration structures that arise. This is mostly standard material on double category theory, found for example in \cite{Framedbicats,GrandisDouble}. We then discuss parallel limits in double categories, local presentability in double categories as well as the categories of monads and comonads, which are the main objects of interest in this work. Some of this material is drawn from \cite{SweedlerDouble} -- which provides further particular references -- and some parts are new. 
We also present a number of examples to illustrate the theory.

\subsection*{Double and fibrant double categories}

We here recall certain standard definitions, mostly in order to set terminology. 
\begin{defi}\label{def:doublecats}
  A \emph{(pseudo) double category} $\dc{D}$
  consists of a category $\dc{D}_0$ of objects and
  a category $\dc{D}_1$ of arrows, with identity, source and target, and composition structure
  given by the functors
  \begin{displaymath}
    \B{1}\colon\dc{D}_0\to\dc{D}_1,\quad
    \Gr{s},\Gr{t}\colon\dc{D}_1\rightrightarrows\dc{D}_0,\quad \text{and}\quad
    \odot\colon\dc{D}_1{\times_{\dc{D}_0}}\dc{D}_1\to\dc{D}_1
  \end{displaymath}
  respectively such that
  $\Gr{s}(1_X)$=$\Gr{t}(1_X)$=$X,\;\Gr{s}(M\odot N)$=$\Gr{s}(N),\;
  \Gr{t}(M\odot N)$=$\Gr{t}(M)$
  for all $X\in\ob(\dc{D}_0)$ and $M,N\in\ob(\dc{D}_1)$,
  equipped with natural isomorphisms
  $a\colon(M\odot N)\odot P\overset{\sim}{\Rightarrow} M\odot(N\odot P)$, $\ell\colon1_{\Gr{s}(M)}\odot M\overset{\sim}{\Rightarrow} M$, $r\colon M\odot1_{\Gr{t}(M)}\overset{\sim}{\Rightarrow} M$
  in $\dc{D}_1$ such that
  $\Gr{t}(a),\Gr{s}(a),\Gr{t}(\ell),\Gr{s}(\ell),
  \Gr{t}(r),\Gr{s}(r)$ are all
  identities, and satisfying the usual coherence conditions.
\end{defi}

The objects of $\dc{D}_0$ are the \emph{0-cells} and its morphisms are the
\emph{vertical} or \emph{tight 1-cells} $f\colon X\to Y$. The objects of $\dc{D}_1$ are the \emph{horizontal} or \emph{loose 1-cells}
$M\colon X\bular Y$, and the morphisms of $\dc{D}_1$ are the
\emph{2-cells}, which will be drawn as
\begin{equation}\label{2morphism}
  \bultwocell{X}{M}{Y}{g}{W}{N}{Z}{f}{\alpha}
\end{equation}
and sometimes denoted by $^f\alpha^g:M\Rightarrow N$.
Strict (vertical) identities are $\mathrm{id}_X:X\to X$ and $\id_M:M\Rightarrow M$,
and horizontal units are $1_X\colon X\bular X$ and
$1_f:1_X\Rightarrow 1_Y$.
A 2-cell whose vertical 1-cell components are both identity morphisms, such as the structural isomorphisms $a,\ell,r$ above, is called  \emph{globular}.

From any double category $\dc{D}$, we obtain a bicategory $\ca{H}(\dc{D})$ called its \emph{horizontal bicategory}, consisting of the 0-cells, the horizontal 1-cells and only the globular 2-cells between them.
We denote by $^X\dc{D}_1$ and $\dc{D}_1^Y$ the subcategories of $\dc{D}_1$ of fixed-domain or fixed-codomain horizontal 1-cells, and 2-cells with identity domain or codomain vertical arrows respectively, namely
\begin{displaymath}
  \begin{tikzcd}
    X\ar[r,bul,"M"]\ar[d,equal]\ar[dr,phantom,"\Two"] & Y\ar[d,"g"] \\
    X\ar[r,bul,"N"'] & W
  \end{tikzcd}\qquad\mathrm{or}\qquad
  \begin{tikzcd}
    X\ar[r,bul,"M"]\ar[d,"f"']\ar[dr,phantom,"\Two"] & Y\ar[d,equal] \\
    Z\ar[r,bul,"N"'] & Y.
  \end{tikzcd}
\end{displaymath}
Similarly, $^X\dc{D}^Y_1=\ca{H}(\dc{D})(X,Y)=\dc{D}[X,Y]$ denotes the subcategory of fixed-domain and fixed-codomain horizontal 1-cells together with globular 2-cells, namely the hom-categories of its horizontal bicategory.

We denote by $\End(\dc{D})$ the subcategory of $\dc{D}_1$ of horizontal endo-1-cells $M:X\bular X$ and 2-cells with the same source and target denoted $a_f$, namely the equalizer of $\mathfrak{s},\mathfrak{t}:\dc{D}_1\rightrightarrows \dc{D}_0$. As it will be useful to what follows, notice that as for all equalizers, $\End(\dc{D})$ can be equivalently written as the pullback
\[
  \begin{tikzcd}
    \End(\dc{D}) \arrow[r] \arrow[d] & \dc{D}_1 \arrow[d, "{\langle \mathfrak{s},\mathfrak{t}\rangle}"] \\
    \dc{D}_0 \arrow[r, "\Delta"']    & \dc{D}_0\times\dc{D}_0.
  \end{tikzcd}
\]

\begin{ex}\label{ex:Span}
  Given a category $\ca{C}$ which has pullbacks, there is a double category $\dc{S}\nc{pan}(\ca{C})$ whose vertical category $\dc{S}\nc{pan}(\ca{C})_0$ is $\ca{C}$ itself with horizontal 1-cells $M\colon X\bular Y$ being spans in $\ca{C}$. A 2-cell $^f\alpha^g:M\Rightarrow N$ is a morphism $\alpha\colon M\to N$ in $\ca{C}$ making the diagram
  \begin{center}
    \begin{tikzcd}
      X\ar[d,"f"'] & M\ar[l]\ar[r]\ar[d,"\alpha"] & Y\ar[d,"g"] \\
      Z & N\ar[l]\ar[r] & W
    \end{tikzcd}
  \end{center}
  commute. Horizontal composition is given by pullback, which is of course associative only up to (coherent) isomorphism, while for every object $X\in\ca{C}$ the 2-cell $1_X$ is the span of identities $
  \begin{tikzcd}
    X & X\ar[l,"\rm{id}_X"']\ar[r,"\rm{id}_X"] & X.
  \end{tikzcd}$
Notice how the category of horizontal endo-1-cells $\End(\Span(\ca{C}))$ is $\nc{Grph}(\ca{C})$, the category of internal graphs in $\ca{C}$. 
\end{ex}

\begin{ex}\label{ex:Rel}
  If $\ca{C}$ is a category with pullbacks which is furthermore regular, then there is also a double category $\dc{R}\nc{el}(\ca{C})$ whose horizontal 1-cells are relations in $\ca{C}$. Horizontal composition is then given by the usual composition of relations in any regular category: after performing the pullback as in $\dc{S}\nc{pan}(\ca{C})$ above, one takes the monic part of the (regular epi, mono)-factorization of the resulting span (viewed as a morphism into a binary product). The unit 1-cell $1_X\colon X\bular X$ is the diagonal relation $\Delta_X=\langle 1_X,1_X\rangle\colon X\rightarrowtail X\times X$.
  Note that now a 2-cell $^f\alpha^g:M\Rightarrow N$ can be described as a pair of morphisms $(f,g)$ in $\ca{C}$ such that we have an inclusion of relations $gMf^{\circ}\subseteq N$ (or equivalently $M\subseteq g^{\circ}Nf$), since the morphism $\alpha\colon M\to N$ is uniquely determined. For $\ca{C}=\nc{Set}$ this is just the implication $(x,y)\in M\implies (f(x),g(y))\in N$. 
\end{ex}

\begin{ex}
  Similarly, if $\ca{C}$ is a category enriched in the cartesian closed category $\nc{Pos}$ of posets and order preserving maps, which is regular in the sense of \cite{Kurz-Velebil}, then we can define a double category $\dc{R}\nc{el}_{\nc{Idl}}(\ca{C})$ whose horizontal 1-cells are the \emph{ideals}. The latter are internal relations $R\colon X\bular Y$ (here meaning a jointly \emph{fully faithful} pair) which satisfy the implication
  \begin{displaymath}
    x'\leq x, \quad (x,y)\in_A R, \quad y\leq y'\quad \implies (x',y')\in_A R
  \end{displaymath}
  for all generalized elements $x,x'\colon A\to X$ and $y,y'\colon A\to Y$ in $\ca{C}$. Composition is defined analogously to the ordinary case, now with respect to the $(\nc{so},\nc{ff})$ factorization system which is present in such a setting (see \cite{VassilisPosExReg}). The units $1_X\colon X\bular X$ are the canonical ideals $I_X$ obtained as comma squares of identity morphisms
  \begin{displaymath}
    \begin{tikzcd}
      I_X\ar[r]\ar[d]\ar[dr,phantom,"\leq"] & X\ar[d,equal] \\
      X\ar[r,equal] & X
    \end{tikzcd}
  \end{displaymath}
We note that in the case $\ca{C}=\nc{Pos}$ this is the double category often denoted by $\dc{P}\nc{os}$ (e.g. \cite{NiefieldSpan}).
\end{ex}

\begin{ex}\label{ex:VMMat}
  If $\ca{V}$ is any monoidal category with small coproducts which are preserved by $\mi\otimes\mi$ in both variables (for example, if $\ca{V}$ is symmetric monoidal closed), then there is a double category $\VMMat$ of $\ca{V}$-matrices.
  The vertical category $\VMMat_0$ is just $\nc{Set}$, whereas
  the horizontal 1-cells $S\colon X\tickar Y$ are functors $S\colon Y\times X\to\ca{V}$, where $Y\times X$ is viewed
  as a discrete category; equivalently, these are families of objects $\{S(y,x)\}_{(y,x)\in Y\times X}$ in $\ca{V}$,
  sometimes also denoted $\{S_{y,x}\}$. The 2-morphisms $^f\alpha^g\colon S\Rightarrow T$ are natural transformations
  \begin{displaymath}
    \begin{tikzcd}[row sep=.1in,baseline=2ex]
      Y\times X\ar[rr,bend left,"S"]\ar[rr,phantom,"\Two\alpha"]\ar[dr,bend right=5,"g\times f"'] && \ca{V} \\
      & Z\times W\ar[ur,bend right=5,"T"'] &
    \end{tikzcd}
  \end{displaymath}
  given by families of arrows $\alpha_{y,x}\colon S(y,x)\to T(g(y),f(x))$ in $\ca{V}$, for all $x\in X$ and $y\in Y$.
  The identity $\ca{V}$-matrix $1_X\colon X\tickar X$  is given by $1_X(x',x)=I$ if $x=x'$ or $0$ otherwise,
  and the horizontal composition functor maps two composable $\ca{V}$-matrices $T\colon Y\tickar Z$ and $S\colon X\tickar Y$ to
  $T\circ S\colon X\tickar Z$ given by
  \begin{displaymath}
    (T\circ S)(z,x)=\sum_{y\in Y} T(z,y)\otimes S(y,x).
  \end{displaymath}
  For more on enriched matrices, see for example \cite{Varthrenr,VCocats}. In particular, the category of endomaps is the usual category of enriched graphs, $\End(\VMMat)=\VGrph$.
\end{ex} 

\begin{ex}
  There is a double category $\dc{R}\nc{ing}$ of rings and bimodules, see e.g. \cite{Pare2021morphisms}.
  Specifically, $\dc{R}\nc{ing}_0=\nc{Ring}$ is the category of (unital) rings, a horizontal 1-cell $R\bular S$ is an $(S,R)$-bimodule and the horizontal composition of $M\colon R\bular S$ and $N\colon S\bular T$ is given by the tensor product $N\otimes_{S} M$. 
  The unit $1_R\colon R\bular R$ is simply the ring $R$ equipped with its canonical structure as an $(R,R)$-bimodule. A 2-cell
  \begin{displaymath}
    \begin{tikzcd}
      R\ar[r,bul,"M"]\ar[d,"f"']\ar[dr,phantom,"\Two\phi"] & S\ar[d,"g"] \\
      R'\ar[r,bul,"{M'}"'] & S'
    \end{tikzcd}
  \end{displaymath}
  is homomorphism $\phi\colon M\to M'$ of $(S,R)$-bimodules, where $M'$ is viewed as such by restriction of scalars on the left and right along $g$ and $f$ respectively. In other words, $\phi$ is an abelian group homomorphism which satisfies $\phi(sm)=g(s)\phi(m)$ and $\phi(mr)=\phi(m)f(r)$
for all $m\in M$, $r\in R$ and $s\in S$.

We should mention here that $\dc{R}\nc{ing}$ is a special case of the general construction \cite{Framedbicats} of the double category $\dc{B}\nc{im}(\dc{D})$, or $\dc{M}\nc{od}(\dc{D})$, of monads and bimodules in $\dc{D}$\footnote{Of course, the bimodule construction also captures many other interesting examples of double categories. For instance, given a cocomplete monoidal closed category $\ca{V}$, $\dc{B}\nc{im}(\VMMat)$ is the double category $\ca{V}\mi\dc{P}\nc{rof}$ of $\ca{V}$-enriched categories and profunctors (or distributors). Correspondingly, for a category $\ca{C}$ with pullbacks and universal coequalizers, $\dc{B}\nc{im}(\dc{S}\nc{pan}(\ca{C}))$ is the double category $\dc{P}\nc{rof}(\ca{C})$ of internal categories and distributors.}, where $\dc{D}$ is required to have \emph{stable local} coequalizers in our terminology (\cref{def: stable local}). Specifically, $\dc{R}\nc{ing}$ is obtained by taking $\dc{D}$ to be the double category $\dc{A}\nc{b}$ of abelian groups; the latter is just the monoidal category $\nc{Ab}$ considered as an one-object, only identity vertical arrow double category. 
\end{ex}

\begin{ex}\label{ex:Poly}
  For a locally cartesian closed category  $\ca{C}$, the double category $\dc{P}\nc{oly}(\ca{C})$ of polynomials over $\ca{C}$ introduced in \cite{GambinoKockPoly}, has 0-cells and vertical 1-cells the objects and arrows of $\ca{C}$ respectively, i.e. $\dc{P}\nc{oly}(\ca{C})_0=\ca{C}$. Its horizontal 1-cells are \textit{polynomials} over $\ca{C}$, namely diagrams of the form
  \[
    I\leftarrow X\to Y\to J
  \]
  in $\ca{C}$ which induce polynomial endofunctors of $\ca{C}$. A 2-cell between two polynomials $I\leftarrow X\to Y\to J$ and $I'\leftarrow X'\to Y'\to J'$ is a commutative diagram of the form
  \[
  \begin{tikzcd}
	I & X & Y & J \\
	& \bullet & Y \\
	{I'} & {X'} & {Y'} & {J'}
	\arrow[from=1-1, to=3-1]
	\arrow[from=1-2, to=1-1]
	\arrow[from=1-2, to=1-3]
	\arrow[from=1-3, to=1-4]
	\arrow[from=1-3, to=2-3, equal]
	\arrow[from=1-4, to=3-4]
	\arrow[from=2-2, to=1-2]
	\arrow[from=2-2, to=2-3]
	\arrow[from=2-2, to=3-2]
	\arrow["\lrcorner"{anchor=center, pos=0.125}, draw=none, from=2-2, to=3-3]
	\arrow[from=2-3, to=3-3]
	\arrow[from=3-2, to=3-1]
	\arrow[from=3-2, to=3-3]
	\arrow[from=3-3, to=3-4]
\end{tikzcd}
\]
which represents a natural transformation between the induced polynomial functors. Composition of $I\leftarrow X\to Y\to J$ and $J\leftarrow X'\to Y'\to J'$ is the polynomial $I\leftarrow M\to N\to J'$ obtained by taking pullbacks and distributivity pullbacks \cite[Def.~2.2.2]{WeberPoly}, as follows
  \[
\begin{tikzcd}[row sep=8pt, column sep=12pt]
	&&& M && N && \bullet &&& \\
	&&&& {} & \bullet & {\text{\footnotesize d.pb}} \\
	&& X && Y && {X'} && {Y'} \\
	I &&&&& J &&&&& {J'}
	\arrow[from=1-4, to=1-6]
	\arrow[from=1-4, to=3-3]
	\arrow["\lrcorner"{anchor=center, pos=0.125}, draw=none, from=1-4, to=3-5]
	\arrow[from=1-6, to=1-8]
	\arrow[from=1-6, to=2-6]
	\arrow[from=1-8, to=3-9]
	\arrow[from=2-6, to=3-5]
	\arrow[from=2-6, to=3-7]
	\arrow["\lrcorner"{anchor=center, pos=0.125, rotate=-45}, draw=none, from=2-6, to=4-6]
	\arrow[from=3-3, to=3-5]
	\arrow[from=3-3, to=4-1]
	\arrow[from=3-5, to=4-6]
	\arrow[from=3-7, to=3-9]
	\arrow[from=3-7, to=4-6]
	\arrow[from=3-9, to=4-11]
\end{tikzcd}
  \]
  We also note another variation, that of $\dc{P}\nc{oly}_c(\ca{C})$ \cite{Monadsindoublecats} which is the sub-double category of $\dc{P}\nc{oly}(\ca{C})$  with 2-cells diagrams
  \[
    \begin{tikzcd}
      I \arrow[d] & X \arrow[rd, "\lrcorner", phantom, very near start] \arrow[r] \arrow[l] \arrow[d] & Y \arrow[r] \arrow[d]  & J \arrow[d] \\
      I' & X' \arrow[l] \arrow[r]  & Y' \arrow[r] & J'
    \end{tikzcd}
  \]
  whose middle square is a pullback. These represent only cartesian natural transformations between the induced polynomial functors \cite[3.12]{GambinoKockPoly}. 
  
  More generally, we can define the double category of polynomials for any category $\ca{C}$ with pullbacks \cite{WeberPoly} as long as we require polynomials to be $I\leftarrow X\to Y\to J$ with $X\to Y$ exponentiable.
\end{ex}



For the remainder of this paper we shall essentially restrict to the setting of \emph{fibrant} double categories (also called \emph{framed bicategories} \cite{Framedbicats}) whose definition we recall below. One technical reason for this is to have a nice relationship between (co)limits in the categories $\dc{D}_{1}$, $\dc{D}_{0}$ on one hand and the various fixed-domain or codomain 1-cell subcategories on the other. In practice, most double categories of interest satisfy this property: indeed, all the examples mentioned earlier are fibrant.

\begin{defi}
  A \emph{fibrant double category} is a double category $\dc{D}$ such that the functor
  \begin{equation}\label{eq:stbifibration}
    \langle\Gr{s},\Gr{t}\rangle\colon\dc{D}_1\longrightarrow\dc{D}_0\times\dc{D}_0
  \end{equation}
  is a fibration, or equivalently an opfibration.
\end{defi}

This means that vertical 1-cells can be turned into horizontal 1-cells in a canonical way. Explicitly, given a vertical 1-cell $f:X\to Y$ in $\dc{D}$,
a \emph{companion}, resp. \emph{conjoint}, of $f$
is a horizontal 1-cell $\wh{f}\colon X\bular Y$, resp. $\wc{f}\colon Y\bular X$, together with
2-morphisms
\begin{displaymath}
  \begin{tikzcd}[ampersand replacement=\&,sep=.3in]
    X\ar[r,bul,"\wh{f}"]\ar[d,"f"']\ar[dr,phantom,"\Two{p_1}"] \& Y\ar[d,equal] \\
    Y\ar[r,bul,"{1_Y}"'] \& Y
  \end{tikzcd}\qquad
  \begin{tikzcd}[ampersand replacement=\&,sep=.3in]
    X\ar[r,bul,"1_X"]\ar[d,equal]\ar[dr,phantom,"\Two{p_2}"] \& X\ar[d,"f"] \\
    X\ar[r,bul,"{\wh{f}}"'] \& Y
  \end{tikzcd}\;\;\textrm{ resp. }\;\;
  \begin{tikzcd}[ampersand replacement=\&,sep=.3in]
    Y\ar[r,bul,"\wc{f}"]\ar[d,equal]\ar[dr,phantom,"\Two{q_1}"] \& X\ar[d,"f"] \\
    Y\ar[r,bul,"{1_Y}"'] \& Y
  \end{tikzcd}
  \qquad
  \begin{tikzcd}[ampersand replacement=\&,sep=.3in]
    X\ar[r,bul,"1_X"]\ar[d,"f"']\ar[dr,phantom,"\Two{q_2}"] \& X\ar[d,equal] \\
    Y\ar[r,bul,"\wc{f}"'] \& X
  \end{tikzcd}
\end{displaymath}
such that $p_1p_2=1_f$ and $p_1\odot p_2\cong1_{\wh{f}}$, resp. $q_1q_2=1_f$ and $q_2\odot q_1\cong 1_{\wc{f}}$.

Furthermore, one can then verify that there are bijections between 2-cells $^f\alpha^g$ and globular 2-cells of the following forms
\begin{equation}
  \begin{tikzcd}
    X\ar[r,bul,"M"]\ar[d,equal]\ar[drr,phantom,"\Two\wh{\alpha}"] & Y\ar[r,bul,"\wh{g}"] & W\ar[d,equal] \\
    X\ar[r,bul,"\wh{f}"'] & Z\ar[r,bul,"N"'] & W
  \end{tikzcd}\quad
  \begin{tikzcd}
    Z\ar[r,bul,"\wc{f}"]\ar[d,equal]\ar[drr,phantom,"\Two\wc{\alpha}"] & X\ar[r,bul,"M"] & Y\ar[d,equal] \\
    Z\ar[r,bul,"N"'] & W\ar[r,tick,"\wc{g}"'] & Y
  \end{tikzcd}
\end{equation}
\begin{displaymath}
  \begin{tikzcd}
    X\ar[d,equal]\ar[rrr,bul,"M"]\ar[drrr,phantom,"\Two"] &&& Y\ar[d,equal] \\
    X\ar[r,bul,"\wh{f}"'] & Z\ar[r,bul,"N"'] & W\ar[r,bul,"\wc{g}"'] & Y
  \end{tikzcd}
  \quad
  \begin{tikzcd}
    Z\ar[d,equal]\ar[r,bul,"\wc{f}"]\ar[drrr,phantom,"\Two"] & X\ar[r,bul,"M"] & Y\ar[r,bul,"\wh{g}"] & W\ar[d,equal] \\
    Z\ar[rrr,bul,"N"'] &&& W.
  \end{tikzcd}
\end{displaymath}

\begin{ex}
  In the double category $\dc{S}\nc{pan}(\ca{C})$, given any $f\colon X\to Y\in\ca{C}$, the companion and conjoint are respectively the spans
  \begin{displaymath}
    \begin{tikzcd}
      X & X\ar[l,"\rm{id_X}"']\ar[r,"f"] & Y
    \end{tikzcd}
    \quad \text{and}\quad
    \begin{tikzcd}
      Y & X\ar[r,"\rm{id_X}"]\ar[l,"f"'] & X
    \end{tikzcd}
  \end{displaymath}
  The same holds in $\dc{R}\nc{el}(\ca{C})$. In other words, $\wh{f}$ is $f$ viewed as a relation (sometimes called its \emph{graph}), while $\wc{f}$ is precisely the opposite of the latter and is usually denoted by $f^{\circ}$.
\end{ex}

\begin{ex}
  If $\ca{C}$ is a regular $\nc{Pos}$-enriched category and $f\colon X\to Y\in\ca{C}$, then in $\dc{R}\nc{el}_{\nc{Idl}}(\ca{C})$ $\wh{f}$ and $\wc{f}$ are the relations obtained as the following two comma squares:
  \begin{displaymath}
    \begin{tikzcd}
      \wh{f}\ar[d]\ar[r]\ar[dr,phantom,"\leq"] & Y\ar[d,equal] \\
      X\ar[r,"f"'] & Y
    \end{tikzcd}
    \quad
    \begin{tikzcd}
      \wc{f}\ar[d]\ar[r]\ar[dr,phantom,"\leq"] & X\ar[d,"f"] \\
      Y\ar[r,equal] & Y
    \end{tikzcd}
  \end{displaymath}
  If $\ca{C}=\nc{Pos}$, then we have precisely $\wh{f}=\{(x,y)\in X\times Y | f(x)\leq y\}$ and $\wc{f}=\{(y,x)\in Y\times X| y\leq f(x)\}$.
\end{ex}

\begin{ex}
  In the double category $\VMMat$, any function $f\colon X\to Y$ canonically determines two horizontal 1-cells
  $f_*\colon X\tickar Y$ and $f^*\colon Y\tickar X$ given by the $\ca{V}$-matrices
  \begin{displaymath}
    f_*(y,x)=f^*(x,y)=
    \begin{cases}
      I,\quad \mathrm{if  }\;f(x)=y\\
      0,\quad \mathrm{ otherwise}
    \end{cases}
  \end{displaymath}
  where $I$ is the monoidal unit of $\ca{V}$. 
\end{ex}

\begin{ex}
  In $\dc{R}\nc{ing}$, given any ring homomorphism $f\colon R\to S$, one can canonically endow $S$ with an $(S,R)$-bimodule structure by restricting scalars along $f$ on the right, and this defines $\wh{f}$.
  Dually, the conjoint is $S\colon S\bular R$, where $S$ is viewed as an $(R,S)$-bimodule by restricting scalars on the left.
\end{ex}

\begin{ex}
The double categories $\dc{P}\nc{oly}(\ca{C})$ and $\dc{P}\nc{oly}_c(\ca{C})$ are fibrant \cite{GambinoKockPoly}. By similar arguments as in $\dc{S}\nc{pan}(\ca{C})$, for a vertical 1-cell $f:X\to Y$, its companion and conjoint are given in both cases by
  \[
    X\overset{\id_X}{\longleftarrow} X \overset{\id_X}{\longrightarrow} X\overset{f}{\longrightarrow} Y\ \ \ \text{and}\ \ \  Y\overset{f}{\longleftarrow} X\overset{\id_X}{\longrightarrow} X\overset{\id_X}{\longrightarrow} X
  \]
  respectively.
\end{ex}


In what follows, we establish various adjunctions between fibrations that arise in the fibrant double setting. We here provide the statement of a general `fibred adjunction' result which is used repeatedly, and follows from an inspection of the proof of \cite[Prop.~8.4.2]{Handbook2}. Recall that a functor between fibrations is called \emph{fibred} when it makes the formed triangle commute and preserves cartesian liftings.

\begin{prop}\label{thm:totaladjointthm}
  Suppose $H\colon\ca{A}\to\ca{B}$ is a fibred functor between fibrations $F\colon\ca{A}\to \ca{C}$ and $G\colon\ca{B}\to\ca{C}$. Then $H$ has a left adjoint $L\colon\ca{B}\to\ca{A}$ if for every $I\in\ca{C}$, there are `fiberwise' adjunctions $L_I\dashv H_I$.
  \begin{displaymath}
    \begin{tikzcd}
\ca{A}\ar[rr,"H"',"\bot"]\ar[dr,"F"'] && \ca{B}\ar[dl,"G"]\ar[ll,dotted,bend right=20,"L"'] \\
& \ca{C} &
    \end{tikzcd}
  \end{displaymath}
\end{prop}

Its dual statement for a right adjoint of an opfibred functor holds as well, reducing its existence to adjunctions between the fibres. On the other hand, right adjoints for fibred functors and left adjoints for opfibred functors require extra conditions. For a relevant discussion, see e.g. \cite[\S~5.3]{PhDChristina}.

\subsection*{Parallel limits and colimits}
We review some basic material concerning parallel limits and colimits in double categories, see \cite[\S~3.1]{SweedlerDouble}. Notably, these do not coincide with the general double-categorical (co)limits found in e.g. \cite{Limitsindoublecats}, rather correspond to `local' ones; the terminology was introduced by Bob Par\'e. 

\begin{defi}\label{def:parallelcolimits}
  Let $\dc{D}$ be a double category and $\ca{I}$ is any small category. We say that $\dc{D}$ has \emph{parallel
  $\ca{I}$-(co)limits} if the categories $\dc{D}_0$, $\dc{D}_1$ both have $\ca{I}$-(co)limits and the functors $\Gr{s},\Gr{t}\colon\dc{D}_1\to\dc{D}_0$
  preserve them. If $\dc{D}$ has parallel $\ca{I}$-(co)limits for any small $\ca{I}$, then we say that it is \emph{parallel
  (co)complete}.
\end{defi}

\begin{ex}
The double category $\dc{S}\nc{pan}(\ca{C})$ has parallel $\ca{I}$-(co)limits for any $\ca{I}$ for which $\ca{C}$ itself has $\ca{I}$-(co)limits. This is easy to see because $\dc{S}\nc{pan}(\ca{C})_1$ can be viewed as a category of functors into $\ca{C}$.
\end{ex}

\begin{ex}
The double category $\dc{R}\nc{el}(\ca{C})$ has parallel $\ca{I}$-limits whenever $\ca{C}$ itself has $\ca{I}$-limits. These are computed as in $\dc{S}\nc{pan}(\ca{C})$, since the span induced between limits by a diagram of relations is itself jointly monic. It is also true that $\ca{C}$ having $\ca{I}$-colimits yields parallel $\ca{I}$-colimits in $\dc{R}\nc{el}(\ca{C})$. The latter are constructed by first forming the colimit of underlying spans and then taking the image in $\ca{C}$ of the corresponding morphism into the binary product, in order to yield a jointly monic pair.
\end{ex}

\begin{ex}\label{ex:VMatparallelcocomplete}
  The double category $\VMMat$ has parallel coproducts.
  More generally, if $\ca{V}$ has $\ca{I}$-colimits for some small $\ca{I}$, then $\VMMat$ has parallel $\ca{I}$-colimits \cite[Proposition 3.1.10]{SweedlerDouble}.
\end{ex}

\begin{ex}
The double category $\dc{R}\nc{ing}$ is parallel cocomplete. Regarding coequalizers in $\dc{R}\nc{ing}_1$, consider
  \begin{displaymath}
    \begin{tikzcd}[ampersand replacement=\&,sep=.3in]
      A\ar[r,bul,"M"]\ar[d,"f_1"']\ar[dr,phantom,"\Two{\phi_1}"] \& B\ar[d,"g_1"] \\
      R\ar[r,bul,"N"'] \& S
    \end{tikzcd}\qquad
    \begin{tikzcd}[ampersand replacement=\&,sep=.3in]
      A\ar[r,bul,"M"]\ar[d,"f_2"']\ar[dr,phantom,"\Two{\phi_2}"] \& B\ar[d,"g_2"] \\
      R\ar[r,bul,"N"'] \& S,
    \end{tikzcd}
  \end{displaymath}
  Take the coequalizers
  \begin{tikzcd}A\ar[r,shift left=1ex,"f_1"]\ar[r,shift right=1ex,"f_2"'] & R\ar[r,two heads,"p"] & R/I
  \end{tikzcd} and
  \begin{tikzcd}B\ar[r,shift left=1ex,"g_1"]\ar[r,shift right=1ex,"g_2"'] & S\ar[r,two heads,"q"] & S/J
  \end{tikzcd} in $\nc{Ring}$ and then consider the abelian group $\overline{N}\coloneqq N/(JN+NI+S\rm{Im}(\phi_1-\phi_2)R)$ with the quotient projection $\pi\colon N\to\overline{N}$. It is then easy to check that
  \begin{displaymath}
    \begin{tikzcd}[ampersand replacement=\&,sep=.3in]
      R\ar[r,bul,"N"]\ar[d,"p"']\ar[dr,phantom,"\Two{\pi}"] \& S\ar[d,"q"] \\
      R/I\ar[r,bul,"\overline{N}"'] \& S/J
    \end{tikzcd}
  \end{displaymath}
  is the coequalizer of $(\phi_1,\phi_2)$. Regarding coproducts, for two $M\colon R\bular S$ and $M'\colon R'\bular S'$, consider the coproducts of rings $R+ R'$ and $S+ S'$. Recall that $R+ R'$ is constructed by taking a certain quotient of the tensor algebra $T(R\oplus R')$ on the abelian group $R\oplus R'$. Consider the free $(S+ S',R+ R')$-bimodule $F(M\oplus M')=(S+ S')\otimes_{\mathbb{Z}}(M\oplus M')\otimes_{\mathbb{Z}}(R+ R')$ on the abelian group $M\oplus M'$ and set $M+ M'\coloneq F(M\oplus M')/V$, where $V$ is the sub-bimodule generated by all elements of the following forms:
  \begin{center}
    $
    \begin{cases}
      s\otimes m\otimes 1 - 1\otimes sm\otimes 1, & s\in S,m\in M \\
      s'\otimes m'\otimes 1 - 1\otimes s'm'\otimes 1, & s'\in S',m'\in M' \\
      1\otimes m\otimes r - 1\otimes mr\otimes 1, & r\in R,m\in M \\
      1\otimes m'\otimes r' - 1\otimes m'r'\otimes 1, & r'\in R',m'\in M'
    \end{cases}$
  \end{center}
  Then we have canonical morphisms
  \begin{displaymath}
    \begin{tikzcd}[ampersand replacement=\&,sep=.3in]
      R\ar[r,bul,"M"]\ar[d,"i_R"']\ar[dr,phantom,"\Two{i_M}"] \& S\ar[d,"i_S"] \\
      R+ R'\ar[r,bul,"{M+ M'}"'] \& S+ S'
    \end{tikzcd}\qquad
    \begin{tikzcd}[ampersand replacement=\&,sep=.3in]
      R'\ar[r,bul,"{M'}"]\ar[d,"i_{R'}"']\ar[dr,phantom,"\Two{i_{M'}}"] \& S'\ar[d,"i_{S'}"] \\
      R+ R'\ar[r,bul,"{M+ M'}"'] \& S+ S'
    \end{tikzcd}
  \end{displaymath}
  which, as can be easily verified, constitute a coproduct diagram in $\dc{R}\nc{ing}_1$.
\end{ex}

\begin{ex}
Limits and colimits in $\dc{P}\nc{oly}(\ca{C})$ have only been systematically studied in the case of $\dc{P}\nc{oly}(\nc{Set})$. In particular, it has been shown that $\dc{P}\nc{oly}(\nc{Set})_1$ has all small limits \cite[Thm 5.33]{spivakpolybook} and colimits \cite[Thm 5.43]{spivakpolybook} 
Moreover, the source and target functors $\mathfrak{s,t}$ are right adjoints and are thus limit-preserving so we can deduce that  $\dc{P}\nc{oly}(\nc{Set})$ is parallel complete.
\end{ex}

Below, we recall \cite[Prop.~3.1.11\&~3.1.13]{SweedlerDouble} respectively. We give a short alternative proof for the first one, essentially to recall the well-known fact that (op)fibrations can be seen to strictly preserve a choice of (co)limits when they are (co)continuous in the ordinary sense.

\begin{prop}\label{lem:(co)limits in End(D)}
  Let $\dc{D}$ be a fibrant double category which has parallel $\ca{I}$-(co)limits, for some small category $\ca{I}$. Then $\End(\dc{D})$ has
  $\ca{I}$-(co)limits and the inclusion $\End(\dc{D})\to\dc{D}_1$ creates them.
\end{prop}

\begin{proof}
  Suppose $D:\ca{I}\to \End(\dc{D})$ is a diagram in $\End(\dc{D})$. Since $\dc{D}$ has parallel $\ca{I}$-colimits, $\ca{I}\to \End(\dc{D})\hookrightarrow \dc{D}_1$ has a colimit
  \[
    \begin{tikzcd}
      X_i \arrow[bul, r, "D_i" {name=Di}]\ar[dr,phantom,"\Two q_i"] \arrow[d, "\mathfrak{s}(q_i)"'] & X_i \arrow[d, "\mathfrak{t}(q_i)"] \\
      X \arrow[bul, r, "D"' {name=D}] & Y
    \end{tikzcd}
  \]
   We will show that $D$ and $q_i$ are in $\End(\dc{D})$. Since $\dc{D}$ is fibrant, $\langle \mathfrak{s}, \mathfrak{t}\rangle$ is a bifibration and also projections and compositions of bifibrations are bifibrations \cite[Prop.~8.1.12 \& 13]{Handbook2}. Therefore, both composites
  \[
    \begin{tikzcd}
      \dc{D}_1 \arrow[rr, "{\langle\mathfrak{s},\mathfrak{t}\rangle}"] &  & \dc{D}_0\times \dc{D}_0 \arrow[rr, "\pi_1", shift left] \arrow[rr, "\pi_0"', shift right] &  & \dc{D}_0
    \end{tikzcd}
  \]
  are bifibrations and so they both (strictly) preserve colimits. This means that $\mathfrak{s}(q_i)=\mathfrak{t}(q_i)$. 
\end{proof}

\begin{prop}\label{prop:equivalentdefparcom}
  Let $\dc{D}$ be a fibrant double category, and $\dc{D}_0$ be complete. The following are equivalent:
  \begin{enumerate}[(i),ref=\thedefi(\roman*)]
    \item $\dc{D}$ is parallel complete;
    \item $^{X}\dc{D}_{1}$ and $\dc{D}_{1}^{Z}$ are complete categories for any $X,Z\in\dc{D}_0$ (and $\mi\odot\wh{f}\colon
      ^Y\dc{D}_1\to{}^X\dc{D}_1$ and $\wc{g}\odot\mi\colon\dc{D}_1^W\to\dc{D}_1^Z$ are continuous functors for any $f\colon X\to Y$ and $g\colon Z\to W$);
    \item  $\dcD{X}{Z}$ is a complete category for any $X,Z\in\dc{D}_0$ (and $-\odot\wh{f}\colon \dcD{Y}{Z}\to\dcD{X}{Z}$ and $\wc{g}\odot-\colon \dcD{X}{W}\to\dcD{X}{Z}$ are continuous functors).
  \end{enumerate}
\end{prop}


We now recall the following terminology from \cite[Lem. 2.4.1]{CommutingTensorProducts}.

\begin{defi}\label{def: stable local}
  For a double category $\dc{D}$ and a small category $\ca{I}$, we say that $\dc{D}$ has \emph{local $\ca{I}$-(co)limits} if $\dcD{X}{Y}$ has $\ca{I}$-(co)limits for any $X,Y\in\dc{D}_0$. Such (co)limits will be called \emph{stable} if they are preserved by the functors $M\odot-\colon\dcD{Z}{X}\to\dcD{Z}{Y}$ and $-\odot M\colon\dcD{Y}{Z}\to\dcD{X}{Z}$ for any $M\colon X\bular Y$ and $Z\in\dc{D}_0$.
\end{defi}

\begin{rmk}\label{rmk:term_conv}
Observe that a double category with parallel (co)limits of some given type also has local colimits of that type. However, stability of such colimits is a strictly stronger requirement. In the converse direction, local (co)limits can be used to obtain parallel such (co)limits. In fact, in \cref{prop:equivalentdefparcom} the continuity conditions for $\mi\odot\wh{f}$ and $\wc{g}\odot\mi$ can be dropped; hence, it follows that in the fibrant setting existence of (co)limits in $\dc{D}_0$ together with local (co)limits in $\dc{D}$ imply existence of parallel (co)limits in $\dc{D}$, see also \cite[Lem.~2.4.2]{CommutingTensorProducts}. 

Because of this fact and in order to ease the statement of the assumptions in results throughout the paper, we make the the following convention: whenever we say that a double category $\dc{D}$ \emph{has stable local $\ca{I}$-colimits}, this will include the assumption that $\dc{D}_0$ has $\ca{I}$-colimits. Thus, having stable local colimits of a given kind will in particular imply the existence of parallel such colimits.
\end{rmk}

\begin{ex}\label{ex:SpanVMatstable}
The local $\ca{I}$-colimits in $\dc{S}\nc{pan}(\ca{C})$ will be stable precisely when the $\ca{I}$-colimits in $\ca{C}$ are \emph{universal} (=pullback-stable). The same remains true for $\dc{R}\nc{el}(\ca{C})$, 
since the extra step in the construction of colimits involves taking a (regular epi, mono) factorization and the latter are also stable under pullback.
	
The double category $\VMMat$ has local $\ca{I}$-colimits when $\ca{V}$ has, since $\VMMat[X,Y]=[Y\times X,\ca{V}]$. It will have stable local $\ca{I}$-colimits if the tensor product preserves $\ca{I}$-colimits in each variable. In particular, if $\ca{V}$ is cocomplete and monoidal closed, then it has all stable local colimits.
\end{ex}

We end this subsection by discussing a notion of local presentability for fibrant double categories investigated in \cite[\S~3.4]{SweedlerDouble}.

\begin{defi}\label{defi:lp}
  A fibrant double category $\dc{D}$ will be called \emph{locally $\lambda$-presentable}, for a regular cardinal $\lambda$, if both categories $\dc{D}_0$, $\dc{D}_1$, are locally $\lambda$-presentable, the functors $\mathfrak{s},\mathfrak{t}\colon\dc{D}_1\to\dc{D}_0$ are cocontinuous right adjoints and, for every horizontal 1-cell $M:X\bular Y$, the functors $-\odot M:\, ^Y\dc{D}_1\to \, ^X\dc{D}_1$ and $M\odot -: \dc{D}_1^X\to \dc{D}_1^Y$ are accessible. $\dc{D}$ will be called \emph{locally presentable} if it is locally $\lambda$-presentable for some regular cardinal $\lambda$.
\end{defi}

\begin{rmk}\label{rmk:lpstable}
By \cite[Rmk. 3.4.6]{SweedlerDouble}, it is the case that the requirement that the functors $-\odot M:\, ^Y\dc{D}_1\to \, ^X\dc{D}_1$ and $M\odot -\colon \dc{D}_1^X\to \dc{D}_1^Y$ are accessible for any $M:X\bular Y$ is in fact equivalent to the analogous functors between fixed-domain \emph{and} codomain subcategories being accessible, namely $\dc{D}$ having \emph{stable local filtered colimits}, and of course stronger than just asking that $\dc{D}$ has pararallel filtered colimits (\cref{rmk:term_conv}). It is also clear that every locally presentable double category is parallel complete and cocomplete.
\end{rmk}

\begin{ex}\label{ex:lpdouble}
For $\ca{V}$ a locally presentable and monoidal category, where $\otimes$ preserves filtered colimits in each variable, $\VMMat$ is a locally presentable double category, as established in \cite[Prop.~3.4.8]{SweedlerDouble}.

Moreover, $\dc{R}\nc{el}(\Set)$ is locally presentable by \cite[Ex.~3.4.7]{SweedlerDouble}, and $\dc{S}\nc{pan}(\ca{C})$ is locally presentable when $\ca{C}$ itself is locally presentable by \cite[Ex.~3.4.8]{SweedlerDouble}.
\end{ex} 

\begin{cor}\cite[Prop.~3.4.4 and 3.4.9]{SweedlerDouble}\label{thm:End(D) lp}
  Suppose that $\dc{D}$ is a locally presentable double category. Then for every $X,Y\in\dc{D}_0$, the categories $^X\dc{D}_1$, $\dc{D}_1^Y$ and $\dcD{X}{Y}$ are locally presentable, as well as $\End(\dc{D})$.
\end{cor}

\begin{rmk}\label{rmk: on lp defi}
As remarked already in \cite[\S~3.4]{SweedlerDouble}, our running notion of local presentability is not a proper 2-dimensional one. It was initially chosen with an eye towards establishing results in loc. cit., where in particular the various fibers  $^X\dc{D}_1$, $\dc{D}_1^Y$ and $\dcD{X}{Y}$ being locally presentable was a property needed in order to establish the existence of certain adjoint functors. At the same time, this set of conditions could indeed be verified in the examples of interest. In attempting to strike a balance between these two needs, one could of course formulate variations, either weaker or stronger, depending on the circumstances and goals.

In the first direction, one could ostensibly drop the requirement that $\dc{D}_1$ be locally presentable and instead directly require that all categories $^X\dc{D}_1$, $\dc{D}_1^Y$ be locally presentable, thus producing a definition which is essentially a \emph{fiberwise} notion of local presentability of $\dc{D}$. Indeed, we are often eventually working in some fiber of $\dc{D}$ and the fibers are also generally easier to handle, while on the other hand the category $\dc{D}_1$ could be ``too big'' to be locally presentable. However, with this weakening we clearly also lose some potentially desirable properties, such as $\End(\dc{D})$ being locally presentable.

In the other direction, one could consider a definition which is more along the lines of postulating that $\dc{D}$ is an internal category in an appropriate category of locally presentable categories. Explicitly, this would additionally involve the requirements that the unit $1\colon\dc{D}_0\to\dc{D}_1$ and horizontal composition $-\odot-\colon\dc{D}_1\times_{\dc{D}_0}\dc{D}_1\to\dc{D}_1$ preserve filtered colimits. Indeed, we will have to impose a weaker version of these in our main \cref{thm:Mndlp,thm:Cmdlp} below, where horizontal composition is assumed to preserve filtered colimits merely as a functor $-\odot-\colon\End(\dc{D})\to\End(\dc{D})$. In passing, let us note that it is not too hard to see that the stronger definition would indeed imply both this latter condition as well as the conditions in \cref{defi:lp}.

For our current purposes, we keep the formulation of \cref{defi:lp} for local presentability of double categories, asking for different conditions when they are needed and not implied (e.g. \cref{MndColimits}).
\end{rmk}

\cref{thm:End(D) lp} establishes that the category of horizontal endo-1-cells is locally presentable, whereas one of the main goals of this work is to establish that moreover, the categories of monads and comonads are also locally presentable. We now turn to these categories.

\subsection*{Monads and comonads}

Let us now recall the basic structures that this work investigates, namely monads and comonads in double categories.

\begin{defi}\label{Monadindoublecat}
  A \emph{monad} in a double category $\dc{D}$ is a horizontal endo-1-cell $A\colon X\bular X$ equipped with
  \begin{displaymath}
    \begin{tikzcd}[sep=.3in]
      X\ar[r,bul,"A"]\ar[d,equal]\ar[drr,phantom,"\Two\mu"] & X\ar[r,bul,"A"] & X\ar[d,equal] \\
      X\ar[rr,bul,"A"'] && X
    \end{tikzcd}\qquad
    \globtwocell{X}{1_X}{X}{X}{A}{X}{\eta}
  \end{displaymath}
  satisfying the usual associativity
  \[
    \begin{tikzcd}[sep=.3in]
      X \arrow[d, equal] \arrow[r, bul, "A\odot A"] \arrow[dr, phantom, "\Two\mu"] & X \arrow[r, bul, "A"] \arrow[d, equal] \arrow[dr, phantom, "\Two 1_A"] & X \arrow[d, equal] \\
      X \arrow[d, equal] \arrow[r, bul, "A"] \arrow[drr, phantom, "\Two\mu"] & X \arrow[r, bul, "A"] & X \arrow[d, equal]\\
      X \arrow[rr, bul, "A"'] & & X
    \end{tikzcd}=
    \begin{tikzcd}[sep=.3in]
      X \arrow[d, equal] \arrow[r, bul, "A\odot A"] & X \arrow[r, bul, "A"] \arrow[d, phantom, "\Two a"] & X \arrow[d, equal] \\
      X \arrow[d, equal] \arrow[r, bul, "A"] \arrow[dr, phantom, "\Two 1_A"] & X \arrow[d, equal] \arrow[r, bul, "A\odot A"] \arrow[dr, phantom, "\Two \mu"] & X \arrow[d, equal] \\
      X \arrow[drr, phantom, "\Two\mu"] \arrow[d, equal] \arrow[r, bul, "A"] & X \arrow[r, bul, "A"] & X \arrow[d, equal] \\
      X \arrow[rr, bul, "A"'] & & X
    \end{tikzcd}
    \]
    and unit laws
    \[\begin{tikzcd}[sep=.3in]
    X \arrow[r, bul, "1_X"] \arrow[d, equal] \arrow[dr, phantom, "\Two\eta"] & X \arrow[d, equal] \arrow[r, bul, "1_X"] \arrow[dr, phantom, "\Two 1_{\id_X}"] & X \arrow[d, equal]\\ 
    X \arrow[d, equal] \arrow[r, bul, "A"] \arrow[drr, phantom, "\Two\mu"] & X \arrow[r, bul, "A"] & X \arrow[d, equal]\\ 
    X \arrow[rr, bul, "A"] & & X
    \end{tikzcd}=
    \begin{tikzcd}[sep=.3in]
    X \arrow[d, equal] \arrow[r, bul, "1_X"] \arrow[drr, phantom, "\Two\ell"] & X \arrow[r, bul, "A"] & X \arrow[d, equal] \\ 
    X \arrow[rr, bul, "A"] && X,
    \end{tikzcd} \quad
    \begin{tikzcd}[sep=.3in]
    X \arrow[r, bul, "1_X"] \arrow[d, equal] \arrow[dr, phantom, "\Two 1_{\id_X}"] & X \arrow[d, equal] \arrow[r, bul, "1_X"] \arrow[dr, phantom, "\Two\eta"] & X \arrow[d, equal]\\ 
    X \arrow[d, equal] \arrow[r, bul, "A"] \arrow[drr, phantom, "\Two\mu"] & X \arrow[r, bul, "A"] & X \arrow[d, equal]\\ 
    X \arrow[rr, bul, "A"] & & X
    \end{tikzcd}=
    \begin{tikzcd}[sep=.3in]
    X \arrow[d, equal] \arrow[r, bul, "A"] \arrow[drr, phantom, "\Two r"] & X \arrow[r, bul, "1_X"] & X \arrow[d, equal] \\ 
    X \arrow[rr, bul, "A"] && X.
    \end{tikzcd}
  \]
  A \emph{monad morphism} is a 2-morphism ${}^f\alpha^f\colon A\Rightarrow B$ where
  \begin{equation}\label{monadhom}
    \begin{tikzcd}[sep=.3in]
      X\ar[r,bul,"A"]\ar[d,"f"']\ar[dr,phantom,"\Two\alpha"] & X\ar[r,bul,"A"]\ar[d,"f"]\ar[dr,phantom,"\Two\alpha"] & X\ar[d,"f"] \\
      Y\ar[r,bul,"B"'] \ar[d,equal]\ar[drr,phantom,"\Two\mu"] & Y\ar[r,bul,"B"'] & Y\ar[d,equal] \\
      Y\ar[rr,"B"'] && Y
    \end{tikzcd}=
    \begin{tikzcd}[sep=.3in]
      X\ar[r,bul,"A"]\ar[d,equal]\ar[drr,phantom,"\Two\mu"] & X\ar[r,bul,"A"] & X\ar[d,equal] \\
      X\ar[rr,bul,"A"'] \ar[d,"f"']\ar[drr,phantom,"\Two\alpha"] && X\ar[d,"f"] \\
      Y\ar[rr,"B"'] && Y,
    \end{tikzcd}\qquad\qquad
    \begin{tikzcd}[sep=.3in]
      X\ar[r,bul,"1_X"]\ar[d,equal]\ar[dr,phantom,"\Two\eta"] & X\ar[d,equal] \\
      X\ar[r,bul,"A"']\ar[d,"f"']\ar[dr,phantom,"\Two\alpha"] & X\ar[d,"f"] \\
      Y\ar[r,bul,"B"'] & Y
    \end{tikzcd}=
    \begin{tikzcd}[sep=.3in]
      X\ar[r,bul,"1_X"]\ar[d,"f"']\ar[dr,phantom,"\Two1_f"] & X\ar[d,"f"] \\
      Y\ar[r,bul,"1_Y"']\ar[d,equal]\ar[dr,phantom,"\Two\eta"] & Y\ar[d,equal] \\
      Y\ar[r,bul,"B"'] & Y.
    \end{tikzcd}
  \end{equation}
  Dually, a \emph{comonad} in $\dc{D}$ is an endo-1-cell $C\colon Z\bular Z$ equipped with globular 2-morphisms $\delta\colon C\Rightarrow C\odot
  C$ and $\epsilon\colon C\Rightarrow1_U$
  satisfying the usual coassociativity and counit axioms.
  A \emph{comonad morphism} is a 2-morphism ${}^f\alpha^f\colon C\Rightarrow D$
  respecting the comultiplications and counits.
\end{defi}

We obtain a category $\Mnd(\dc{D})$, which is in fact the vertical category of a double category of monads introduced in
\cite[Def.~2.4]{Monadsindoublecats}, as well as a category $\Cmd(\dc{D})$. These are both subcategories of $\End(\dc{D})$ since vertical boundary maps of (co)monad maps coincide.

\begin{ex}\label{ex:spanmonads}
It is well-known that a monad in $\dc{S}\nc{pan}(\ca{C})$ is a category internal to $\ca{C}$. It consists of a horizontal endo-1-cell, i.e. a span of the form $X_0\overset{s}{\leftarrow}X_1\overset{t}{\to} X_0$; $X_0$ will be the intended object of objects and $X_1$ the intended object of arrows; the left and right arrows $s$ and $t$ are the source and target respectively. Furthermore, it comes equipped with a composition $\mu:X_1\times_{X_0} X_1\to X_1$ for compatible arrows which is assocative and a unit $1:X_0\to X_1$ which assigns the identity arrow to every object. The axioms for the unit imply that it is a neutral element for the composition. Finally, a monad morphism then becomes precisely an internal functor, so that $\Mnd(\dc{S}\nc{pan}(\ca{C}))$ is $\nc{Cat}(\ca{C})$. 

  A comonad on the other hand is just an arrow in $\ca{C}$. Indeed, for a horizontal endo-1-cell $C:X\overset{f}{\leftarrow}Y\overset{g}{\to} X$, the existence of a counit $\varepsilon:C\Rightarrow 1_X$ implies that $f=g=\varepsilon$ while the existence of a comultiplication imply that $\Delta$ is just the diagonal; the axioms then become trivial equations.
\end{ex}

\begin{ex}
A monad in $\dc{R}\nc{el}(\ca{C})$ is a reflexive and transitive internal relation, i.e. an internal preorder. Indeed, the existence of 2-cells $\eta:1_X\to R$ and $\mu:R\odot R\to R$ describes the inclusions $\Delta_X\subseteq R$ and $RR\subseteq R$, respectively. In addition, a monad morphism is clearly just an order preserving morphism between internal preorders, so that $\Mnd(\dc{R}\nc{el}(\ca{C}))$ is $\nc{Preord}(\ca{C})$.

  On the other hand, a comonad $C:X\bular X$ is just a subobject of $X$ in $\ca{C}$. The existence of a counit $\varepsilon :C\Rightarrow 1_X$ means that $C\subseteq \Delta_X$, which implies that the two legs of the span representing $C$ are equal. 
\end{ex}

\begin{ex}\label{ex:cocat}
A monad in $\VMMat$ is a category enriched in $\ca{V}$: it is a $\ca{V}$-matrix $A\colon X\tickar X$ namely a family $\{A(x',x)\}$ of objects in $\ca{V}$, together with morphisms $A(x',x'')\ot A(x,x')\to A(x,x'')$ and $I\to A(x,x')$ in $\ca{V}$ satisfying the usual associativity and unitality conditions. Moreover, a monad map in $\VMMat$ is precisely a $\ca{V}$-functor, capturing $\VCat$ as $\Mnd(\VMMat)$. Dually, the category of comonads in matrices is $\ca{V}$-$\mathsf{Cocat}$ of $\ca{V}$-\emph{cocategories} and $\ca{V}$-\emph{cofunctors}: for more details, see \cite{VCocats}.
\end{ex}

\begin{ex}\label{ex:mndBim}
  $\Mnd(\dc{R}\nc{ing})$ is a global category $\Alg$ of algebras over all rings. In more detail, a monad is an $(R,R)$-bimodule $M$ for some ring $R$ together with two $(R,R)$-bimodule homomorphisms $\mu:M\otimes_R M\to M$ and $\eta:R\to M$. The axioms for $\mu$ and $\eta$ imply that $\mu$ is an associative binary operation on $M$ and $\eta(1_R)$ is a unit. Dually, $\Cmd(\dc{R}\nc{ing})$ is a global category of coalgebras over all rings.
\end{ex}

\begin{ex}
For the case where $\ca{C}=\nc{Set}$, the slice categories of $\Mnd(\dc{P}\nc{oly}(\nc{Set}))$ over some polynomial monad $P$ give  $P$-multicategories \cite[Cor. 5.17]{GambinoKockPoly}. In particular, the slice over the monoid monad $M$ is the category of coloured non-symmetric operads and the slice over the identity monad give small categories. 
\end{ex}

In the fibrant setting, categories of horizontal endo-1-cells, monads and comonads obtain certain (op)fibration structures over $\dc{D}_0$, see \cite[Prop.~3.3]{Monadsindoublecats} or
\cite[Prop. 3.3.4]{SweedlerDouble}.

\begin{prop}\label{prop:MonComonfibred}
  Suppose $\dc{D}$ is a fibrant double category. Then $\End(\dc{D})$ is bifibred over $\dc{D}_0$, $\Mnd(\dc{D})$ is fibred over $\dc{D}_0$ and $\Cmd(\dc{D})$ is opfibred over $\dc{D}_0$.
\end{prop}

\begin{proof}
We briefly recall the related structures. For $\End(\dc{D})$, the fiber above some $X\in \dc{D}_0$ is the endo-hom category $\ca{H}(\dc{D})(X,X)=\dcD{X}{X}$ and for any vertical arrow $f:X\to Y$, the reindexing functors are
\begin{equation}\label{eq:adj}
    \begin{tikzcd}
      \dcD{X}{X} \arrow[rr, "{\hat f\odot\, -\, \odot\check f}", "\bot"', shift left=2] &  & \dcD{Y}{Y}. \arrow[ll, "{\check f\odot\, -\, \odot \hat f}", shift left=2]
    \end{tikzcd}
\end{equation}
For $\Mnd(\dc{D})$ and $\Cmd(\dc{D})$, 
the fiber over some $X\in\dc{D}_0$ is the category of (co)monoids over the monoidal category $(\dcD{X}{X},\odot, 1_X)$ with the same reindexing functors, which can be seen to restrict to those categories accordingly.
\end{proof}

Very similarly to \cref{lem:(co)limits in End(D)}, parallel limits in fibrant double categories induce limits in $\Mnd(\dc{D})$ and dually for $\Cmd(\dc{D})$, see \cite[Prop.~3.3.6]{SweedlerDouble}.

\begin{prop}\label{(co)limits in (co)monads}
  If a fibrant double category $\dc{D}$ has parallel $\ca{I}$-limits, then the category $\Mnd(\dc{D})$ has all $\ca{I}$-limits and the inclusion $\Mnd(\dc{D})\to\End(\dc{D})$ creates them. Dually, if $\dc{D}$ has parallel $\ca{I}$-colimits, then the category $\Cmd(\dc{D})$ has $\ca{I}$-colimits and the inclusion $\Cmd(\dc{D})\to \End(\dc{D})$ creates them.
\end{prop}

\begin{rmk}\label{rem:connected}
It is known (e.g. \cite[Cor.~2.2.3]{SweedlerDouble}) that for any fibration, the inclusion of a fibre into the total category preserves any \emph{connected} limit that the fibration preserves, and dually for colimits in opfibrations. This is worth mentioning, because it clarifies that in general, limits in the fibre are not the same as limits in the total category even in the the case of `fibred complete' fibrations. 
In the above case, when $\dc{D}$ has parallel connected limits and colimits, the inclusion of the fiber $\dcD{X}{X}\to\End(\dc{D})$ in the total category of the (op)fibration $\End(\dc{D})\to\dc{D}_0$ preserves all connected limits and colimits, and similarly for monads and comonads. The same holds under the stronger assumption (\cref{rmk:term_conv}) of stable local connected limits and colimits.
\end{rmk}

\section{Monads}

In this section, we study fundamental properties of the category of monads in a double category. Firstly, we show that $\Mnd(\dc{D})$ is monadic over the category of endomorphisms $\End(\dc{D})$ under
suitable hypotheses on the double category $\dc{D}$ -- \cref{thm:Mnd(D)monadic}. This can be seen as a generalization of the monadicity
of $\VCat$ over $\VGrph$ for a (suitable) monoidal category $\ca{V}$ (\cite{Wolff}), or the similar and more general result concerning $\ca{W}\textrm{-}\nc{Cat}$ over a  bicategory $\ca{W}$ (\cite{Varthrenr}). Second, we prove that $\Mnd(\dc{D})$ is cocomplete, again under suitable assumptions on $\dc{D}$. In fact, part of our reasoning establishing this result faithfully adapts arguments of \cite{Varthrenr}, but providing the necessary adaptations to the double-categorical context. Finally, we give conditions under which the category of monads $\Mnd(\dc{D})$ is locally presentable -- \cref{thm:Mndlp}.

The first step, namely the existence of a free monad on an endomorphism in a double category, is based on the standard `free monoid construction' in the setting of a monoidal category; see \cite{FreeMonoids} for a general discussion of relevant conditions.
We briefly recall one version that suits our purposes: if $\ca{V}$ is a monoidal category with countable coproducts which are preserved by $\otimes$ in each variable, then the forgetful functor $\Mon(\ca{V})\to\ca{V}$ has a
left adjoint that maps any $A\in\ca{V}$ to the ``geometric series'' $\sum\limits_{n\in\mathbb{N}}A^{\otimes n}$.

\begin{prop}\label{prop:freemonadfunctorexistence1}
  Suppose that $\dc{D}$ is a fibrant double category with stable local countable coproducts. Then the
  forgetful functor $U\colon\Mnd(\dc{D})\to\End(\dc{D})$ has a left adjoint.
\end{prop}

\begin{proof}
 Since $\dc{D}$ is fibrant, the forgetful functor $U\colon\Mnd(\dc{D})\to\End(\dc{D})$ is a fibred functor
  \begin{displaymath}
    \begin{tikzcd}
      \Mnd(\dc{D})\ar[rr,"U"]\ar[dr] && \End(\dc{D})\ar[dl] \\
      & \dc{D}_{0}
    \end{tikzcd}
  \end{displaymath}
  between the fibrations $\Mnd(\dc{D})\to\dc{D}_0$ and $\End(\dc{D})\to\dc{D}_0$. 
  Indeed, the triangle commutes and the cartesian liftings 
  for both fibrations are essentially the same, 
  see \cref{prop:MonComonfibred}.
  Thus, by \cref{thm:totaladjointthm}, for $U$ to have a left adjoint it suffices to show that for every $X\in\dc{D}_0$, the restriction
  $U_{X}\colon\Mnd(\dc{D})_{X}\to\End(\dc{D})_{X}$ has a left adjoint.

  Indeed, recall that the fiber $\End(\dc{D})_{X}=\dcD{X}{X}$ is a monoidal category with tensor product given by the horizontal
  composition $\odot$ of $\dc{D}$. Moreover, the functor $U_{X}\colon\Mnd(\dc{D})_{X}\to\End(\dc{D})_{X}$ is precisely the forgetful
  functor $\Mon(\dcD{X}{X})\to\dcD{X}{X}$. We know that $\dcD{X}{X}$ has countable coproducts which are preserved by tensoring in each variable by assumption (see \cref{def: stable local}).
  Thus, by the preceding discussion we deduce that under these conditions, $U_X$ has a left adjoint and the result is complete..
\end{proof}

Notice that based in the previous free monoid description, if
$\dc{D}$ is a double category as in the above proposition, then we see that the left adjoint $F\colon\End(\dc{D})\to\Mnd(\dc{D})$ will map any
$G\colon X\bular X\in\End(\dc{D})$ to the geometric series $\sum\limits_{n}{}_{X}G^{\odot n}$, where $\sum_{X}$ denotes the coproduct in the
category $\dcD{X}{X}$. Note that this coproduct is generally not the same as the coproduct in $\End(\dc{D})$ (or $\dc{D}_1$).

As was discussed earlier, the existence of free monoids in the context of a monoidal category $\ca{V}$ can be obtained via different means, depending on the surrounding assumptions on $\ca{V}$ itself. Accordingly, as is clear from the above proof, one could formulate corresponding different sets of conditions on a double category $\dc{D}$ that would guarantee the existence of free monads therein. We record here one such variation concerning the context of a locally presentable double category (\cref{defi:lp}).

\begin{prop}\label{prop:freemonadfunctorexistence2}
  Suppose that $\dc{D}$ is a locally presentable double category. Then the forgetful functor $U\colon\Mnd(\dc{D})\to\End(\dc{D})$ has a left adjoint.
\end{prop}

\begin{proof}
  In the exact same manner to the above proof, in the underlying fibrant setting the problem is reduced to proving the existence of fibrewise adjoints, namely adjoints to $U_{X}\colon\Mon(\dcD{X}{X})\to\dcD{X}{X}$. Since $\dc{D}$ is locally presentable, by \cref{thm:End(D) lp} each fiber $\dcD{X}{X}$ is locally presentable itself and moreover, by \cref{rmk:lpstable}, $\dcD{X}{X}$ has filtered colimits which are preserved by `tensoring' in each variable. Thus, \cite[\S~2.6]{MonComonBimon}\footnote{Although \cite{MonComonBimon} requires that the base monoidal category is symmetric, an inspection of the proofs shows that symmetry can be replaced with the tensor preserving filtered colimits in both entries, rather than in one.} can be applied to yield the required left adjoint.
  
\end{proof}

\begin{rmk}
 We should note here that questions concerning the existence of free monads in double categories have also been considered in \cite{Monadsindoublecats,Doubleadjunctionsandfreemonads}. However, the approach therein differs from ours in at least a couple of aspects.
 
 First of all, the authors of loc. cit. do not investigate conditions on the double category $\dc{D}$ that would directly imply the existence of free monads in $\dc{D}$. 
  Rather, they provide conditions under which, knowing that the horizontal bicategory $\ca{H}(\dc{D})$ admits free monads, one can deduce the same for $\dc{D}$. Second, their notion of ``free monad'' is a 2-dimensional one, which is to say that it must satisfy a stronger universal property than our ordinary one. In particular, they work with a double category $\dc{M}\nc{nd}(\dc{D})$ of monads in $\dc{D}$, whose vertical category is precisely our ordinary category of monads $\Mnd(\dc{D})$.
\end{rmk}

\begin{ex}
	Specializing \cref{prop:freemonadfunctorexistence1} to the double category $\dc{S}\nc{pan}(\ca{C})$ whose monads are internal categories as per \cref{ex:spanmonads}, we have that when $\ca{C}$ has (pullbacks and) universal countable coproducts, then the forgetful functor $\nc{Cat}(\ca{C})\to\nc{Grph}(\ca{C})$ from internal categories to internal graphs in $\ca{C}$ has a left adjoint. Alternatively, going via \cref{prop:freemonadfunctorexistence2}, it can be shown that said left adjoint exists as well when $\ca{C}$ is locally presentable, see \cref{ex:lpdouble}. 
\end{ex}

\begin{ex}
	Let $\ca{C}$ be a regular category which has universal countable coproducts. Then $\dc{R}\nc{el}(\ca{C})$ has stable local countable coproducts and so \cref{prop:freemonadfunctorexistence1} applies. 
	It is easy to see that in this case the left adjoint forms the usual ``reflexive-transitive hull'' of any given relation $R\colon X\bular X$ on an object $X$ in $\ca{C}$, i.e. $F(R)=\bigcup\limits_{n=0}^{\infty}R^{n}$.
\end{ex}

\begin{ex}
	For $\VMMat$, \cref{prop:freemonadfunctorexistence1} immediately applies to yield a left adjoint to $\VCat\to\VGrph$, since arbitrary coproducts  stable under tensoring with a fixed variable are already required in $\ca{V}$ in order to define the composition in $\VMMat$ (\cref{ex:VMMat}). 
	In the special case when $\ca{V}$ is symmetric monoidal closed, this recovers a result originally due to H. Wolff \cite[Corollary 2.3]{Wolff}.
\end{ex}

As in the case of monoids in a monoidal category, we have that $\Mnd(\dc{D})$ is monadic over $\End(\dc{D})$ as soon as free monads exist.

\begin{thm}\label{thm:Mnd(D)monadic}
	Let $\dc{D}$ be a fibrant double category such that the forgetful functor $U\colon\Mnd(\dc{D})\to\End(\dc{D})$ has a left adjoint. Then $\Mnd(\dc{D})$ is monadic over $\End(\dc{D})$.
\end{thm}
\begin{proof}
We apply Beck's monadicity theorem. $U\colon\Mnd(\dc{D})\to \End(\dc{D})$ has a left adjoint by hypothesis, and it can be verified to reflect isomorphisms.
Lastly, if a pair or parallel monad maps
\[
\begin{tikzcd}
X \ar[ddrr,phantom,"{\Two \alpha, \beta}"]
\arrow[dd, "f"', shift right] \arrow[dd, "g", shift left] \arrow[rr, bul, "M"] & & X \arrow[dd, "f", shift left] \arrow[dd, "g"', shift right] \\
& & &  \\
Y \arrow[rr, bul, "N"'] & & Y
\end{tikzcd}
\] 
is such that $U\alpha_f,U\beta_g$ have a split coequalizer 
$\gamma_e\colon N_Y\Rightarrow E_Z$ in $\End(\dc{D})$, then we can give $E$ a monad structure that makes the same diagram a coequalizer in $\Mnd(\dc{D})$.

For the multiplication, applying $-\odot -$ to the above absolute colimit we get a (split) coequalizer $\gamma\odot\gamma$   
of $\alpha\odot \alpha, \beta\odot \beta$ in $\End(\dc{D})$. But $\gamma\circ \mu_N$ also coequalizes this pair via
\begin{displaymath}
\begin{tikzcd}
M\odot M\ar[r,shift left,"\alpha\odot\alpha"]\ar[r,shift right,"\beta\odot\beta"']\ar[d,"\mu_M"'] & N\odot N\ar[d,"\mu_N"] \ar[r,"\gamma\odot\gamma"] & E\odot E\ar[d,dashed,"\exists !\mu_E"]\\
M\ar[r,shift left,"\alpha"]\ar[r,shift right,"\beta"'] & N\ar[r,"\gamma"'] & E
\end{tikzcd}
\end{displaymath}
and therefore exists a unique 2-cell $\mu_E\colon E\odot E\Rightarrow E$ that commutes with the appropriate maps, and is moreover globular due to the vertical part of said commutativity -- since $e$ is the coequalizer of $f,g$ in $\dc{D}_0$ thus epic. To prove that $\mu_E$ is associative, 
we verify that $\mu_E\circ (1_E\odot \mu_E)\circ (\gamma\odot\gamma\odot \gamma)=\mu_E\circ (\mu_E\odot 1_E)\circ (\gamma\odot\gamma\odot \gamma)$ and use the fact that $\gamma\odot\gamma\odot \gamma$ is epic.  

For the unit, applying $1_{(-)}\circ \mathfrak{s}$ on the original split coequalizer yields a split coequalizer $1_e$ of $1_f, 1_g$ in $\dc{D}_1$ -- where recall that for example, $1_e$ is the 2-map $1_Y\Rightarrow 1_Z$ with vertical boundaries $e$. With similar arguments as before, $\gamma\circ \eta_N$ also coequalizes $1_f,1_g$ via
\begin{displaymath}
\begin{tikzcd}
1_X\ar[r,shift left,"1_f"]\ar[r,shift right,"1_g"']\ar[d,"\eta_M"'] & 1_Y\ar[d,"\eta_N"]\ar[r,"1_e"] & 1_Z\ar[d,dashed,"\exists !\eta_E"] \\
M\ar[r,shift left,"\alpha"]\ar[r,shift right,"\beta"'] & N\ar[r,"\gamma"'] & E
\end{tikzcd}
\end{displaymath}
thus obtaining a unique 2-cell $\eta_E\colon 1_E\Rightarrow E$ which, as before, commutes with the appropriate maps and is furthermore globular. For the left unitality of $\eta_E$, we 
verify that $\mu_E\circ (\eta_E\odot\id_E)\circ (1_e\odot \gamma)=\ell_E\circ (1_e\odot \gamma)$ and use the fact that $1_e\odot \gamma$ is epic. Similarily for the right unitality.

By their construction, $\mu_E$ and $\eta_E$ make $E$ into a monad and $\gamma_e$ a monad morphism, and then it can be seen that the original coequalizer remains such in $\Mnd(\dc{D})$.
\end{proof}

%

In particular, we have now isolated conditions on a double category under which the category of monads is monadic over the category of endomorphisms.

\begin{cor}\label{cor:monadicity}
	Let $\dc{D}$ be a fibrant double category which satisfies either one of the following conditions:
	\begin{enumerate}
		\item $\dc{D}$ has stable local countable coproducts.
		\item $\dc{D}$ is locally presentable.
	\end{enumerate}
	Then $\Mnd(\dc{D})$ is monadic over $\End(\dc{D})$.
\end{cor}

We next want to tackle the question of cocompleteness of $\Mnd(\dc{D})$, for which purpose it will suffice in our context to construct coequalizers. To accomplish this, we directly adapt the corresponding results in \cite{Varthrenr}, specifically Lemma 4 and Proposition 5 therein.

\begin{lem}\label{lem:BettiCarboniStreetWalters}
  Let $\dc{D}$ be a fibrant double category which has stable local coequalizers. Suppose that $\phi,\psi\colon A\to B$
  are morphisms in $\Mon(\dcD{X}{X},\odot,1_X)$ for an $X\in\dc{D}_0$, and let $\chi\colon B\to C$ be their coequalizer in $\dcD{X}{X}$. Then
  the following are equivalent:
  \begin{enumerate}
    \item C has a monad structure such that $\chi$ is a (globular) monad morphism.
    \item The morphism $\chi\mu_B\colon B\odot B\to C$ coequalizes both pairs of morphisms $
      \begin{tikzcd}A\odot B \ar[r,shift
        left=1ex,"\phi\odot\id_{B}"]\ar[r,shift right=1ex,"\psi\odot\id_{B}"'] & B\odot B
      \end{tikzcd}$ and $
      \begin{tikzcd}B\odot A\ar[r,shift
        left=1ex,"\id_{B}\odot\phi"]\ar[r,shift right=1ex,"\id_{B}\odot\psi"'] & B\odot B.
      \end{tikzcd}$
  \end{enumerate}
  Furthermore, in this case $\chi$ is the coequalizer of $(\phi,\psi)$ in $\Mnd(\dc{D})$. 
  
  In particular, if the initial pair $(\phi,\psi)$ has a common splitting in
  $\dcD{X}{X}$, condition (2) is satisfied\footnote{We thank Ross Street for clarifying this final part of the statement.}.
\end{lem}
\begin{proof}
  We will make use of the following $3\times 3$ diagram. 
  \begin{displaymath}
    \begin{tikzcd}[sep=4em]
      A\odot A\ar[r,shift left=1ex,"\id_{A}\odot\phi"]\ar[r,shift right=1ex,"\id_{A}\odot\psi"']\ar[d,shift left=1ex,"\phi\odot\id_{A}"]\ar[d,shift
      right=1ex,"\psi\odot\id_{A}"'] & A\odot B\ar[rr,two heads,"\id_{A}\odot\chi"]\ar[d,shift left=1ex,"\phi\odot\id_{B}"]\ar[d,shift
      right=1ex,"\psi\odot\id_{B}"'] && A\odot C\ar[d,shift left=1ex,"\phi\odot\id_{C}"]\ar[d,shift right=1ex,"\psi\odot\id_{C}"'] &  \\
      B\odot A\ar[r,shift left=1ex,"\id_{B}\odot\phi"]\ar[r,shift right=1ex,"\id_{B}\odot\psi"']\ar[dd,two heads,"\chi\odot\id_{A}"'] & B\odot
      B\ar[dr,"\mu_B"]\ar[rr,two
      heads,"\id_{B}\odot\chi"]\ar[dd,two heads,"\chi\odot\id_{B}"]  && B\odot C\ar[dr,dashed,"\xi"] & \\
       & & B\ar[rr,"\chi" near start] & & C \\
      C\odot A\ar[r,shift left=1ex,"\id_{C}\odot\phi"]\ar[r,shift right=1ex,"\id_{C}\odot\psi"'] & C\odot B\ar[rr,two heads,"\id_{C}\odot\chi"]  &&
      C\odot C\ar[ur,dashed,"\mu_C"]\ar[from=uu,crossing over,two heads,"\chi\odot\id_{C}"' near start]  & 
    \end{tikzcd}
  \end{displaymath}
  Observe that by our assumptions on $\odot$, all rows and columns are coequalizer diagrams, both in $\dcD{X}{X}$ and $\End(\dc{D})$. The former is included in \cref{def: stable local}
  of stable local coequalizers and the latter since coequalizers are connected colimits, hence the inclusion 
  $\dcD{X}{X}\to\End(\dc{D})$ 
  preserves them by \cref{rem:connected}.
  \vspace{3mm}

  \underline{(1)$\implies$(2):} If $C$ is a monad and $\chi$ is a monad morphism, then we must have $\mu_{C}(\chi\odot\chi)=\chi\mu_{B}$. This equality can be written as
  $\mu_{C}(\id_{C}\odot\chi)(\chi\odot\id_{B})=\chi\mu_{B}$ and similarly as $\mu_{C}(\chi\odot\id_{C})(\id_{B}\odot\chi)=\chi\mu_{B}$, using the interchange law due to globularity of $\chi$. From these, by precomposing with both indicated pairs, it follows immediately that $\chi\mu_{B}$ coequalizes them.
  \vspace{3mm}

  \underline{(2)$\implies$(1):} First, by the coequalizer property of the second row, there is a unique $\xi\colon B\odot C\to C\in\dcD{X}{X}$
  such that $\xi(\id_{B}\odot\chi)=\chi\mu_{B}$. Now the equality $\chi\mu_{B}(\phi\odot\id_{B})=\chi\mu_{B}(\psi\odot\id_{B})$, which we have by
  assumption, can be rewritten as $\xi(\id_{B}\odot\chi)(\phi\odot\id_{B})=\xi(\id_{B}\odot\chi)(\psi\odot\id_{B})$ and then further as
  $\xi(\phi\odot\id_{C})(\id_{A}\odot\chi)=\xi(\psi\odot\id_{C})(\id_{A}\odot\chi)$. But notice that $\id_{A}\odot\chi$ is an epimorphism in $\dcD{X}{X}$ since $A\odot-$ preserves coequalizers by assumption, so that we can deduce that $\xi(\phi\odot\id_{C})=\xi(\psi\odot\id_{C})$. Now using the coequalizer property of the third
  column, $(\exists!\mu_{C}\colon C\odot C\to C\in\dcD{X}{X})\mu_{C}(\chi\odot\id_{C})=\xi$. We also define $\eta_{C}\coloneqq\chi\eta_{B}$ and
  claim that $(C,\mu_{C},\eta_{C})$ is a monad in $\dc{D}$.

  Note that the proposed definition of $\eta_{C}$ is forced upon us if $\chi$ is to be a monad morphism, as in that case we must have
  $\chi\eta_{B}=\eta_{C}1_{\id_{X}}=\eta_{C}\id_{1_X}=\eta_{C}$. We also have
  $\mu_{C}(\chi\odot\chi)=\mu_{C}(\chi\odot\id_{C})(\id_{B}\odot\chi)=\xi(\id_{B}\odot\chi)=\chi\mu_{B}$. So $\chi$ will indeed be a monad morphism, as
  soon as we show that $C$ is a monad.

  For one of the unit axioms for $C$, we argue as follows:
  \begin{align*}
    \mu_{C}(\eta_{C}\odot\id_{C})(\id_{1_X}\odot\chi)
    &=\mu_{C}((\chi\mu_{B})\odot\id_{C})(\id_{1_X}\odot\chi)=\mu_{C}(\chi\odot\id_{C})(\eta_{B}\odot\id_{C})(\id_{1_X}\odot\chi) \\
    &=\xi(\id_{B}\odot\chi)(\eta_{B}\odot\id_{B})=\chi\mu_{B}(\eta_{B}\odot\id_{B})=\chi\lambda_{B}=\lambda_{C}(\id_{1_X}\odot\chi)
  \end{align*}
  and since $\id_{1_X}\odot\chi$ is an epimorphism, 
  we deduce that
  $\mu_{C}(\eta_{C}\odot\id_{C})=\lambda_{C}$; similarly for the other unit axiom.
  Finally, for the associativity axiom we argue as follows:
  \begin{align*}
    \mu_{C}(\id_{C}\odot\mu_{C})(\chi\odot\chi\odot\chi) &=\mu_{C}(\id_{C}\odot\mu_{C})(\chi\odot\id_{C^{2}})(\id_{B}\odot\chi\odot\chi) \\
    &=\mu_{C}(\chi\odot\id_{C})(\id_{B}\odot\mu_{C})(\id_{B}\odot\chi\odot\chi) \\
    &=\xi(\id_{B}\odot\mu_{C})(\id_{B}\odot\chi\odot\chi)=\xi(\id_{B}\odot\mu_{C}(\chi\odot\chi)) \\
    &=\xi(\id_{B}\odot\chi\mu_{B})=\xi(\id_{B}\odot\chi)(\id_{B}\odot\mu_{B})=\chi\mu_{B}(\id_{B}\odot\mu_{B}) \\
    &=\chi\mu_{B}(\mu_{B}\odot\id_{B})=\mu_{C}(\chi\odot\chi)(\mu_{B}\odot\id_{B}) \\
    &=\mu_{C}(\chi\mu_{B}\odot\chi)=\mu_{C}(\mu_{C}(\chi\odot\chi)\odot\chi) \\
    &=\mu_{C}(\mu_{C}\odot\id_{C})(\chi\odot\chi\odot\id_{C})(\id_{B^{2}}\odot\chi) \\
    &=\mu_{C}(\mu_{C}\odot\id_{C})(\chi\odot\chi\odot\chi)
  \end{align*}
  and again observe that $\chi\odot\chi\odot\chi$ is an epimorphism by the assumptions on $\odot$.
  
  In order to check that $C$ remains the coequalizer in $\Mnd(\dc{D})$, suppose that 
  \begin{displaymath}
    \begin{tikzcd}
      X\ar[r,bul,"B"]\ar[d,"b"']\ar[dr,phantom,"\Two\beta"] & X\ar[d,"b"] \\
      Y\ar[r,bul,"D"'] & Y
    \end{tikzcd}
  \end{displaymath}
  is a morphism in $\Mnd(\dc{D})$ such that $\beta\phi=\beta\psi$. By the coequalizer property in $\End(\dc{D})$ there is a unique $\gamma\colon
  C\to D\in\End(\dc{D})$ such that $\gamma\chi=\beta$. It is then easy to verify that $\gamma$ is a monad morphism, using that $\beta$ is such and that $\chi\odot\chi$ is an epimorphism.


  Finally, assume there exists a splitting namely $\sigma\colon B\to A\in\dcD{X}{X}$ such that $\phi\sigma=\id_{B}=\psi\sigma$. Then 
  \begin{align*}
    \chi\mu_{B}(\phi\odot\id_{B}) &=\chi\mu_{B}(\phi\odot(\phi\sigma))=\chi\mu_{B}(\phi\odot\phi)(\id_{A}\odot\sigma) \\
    &=\chi\phi\mu_{A}(\id_{A}\odot\sigma)=\chi\psi\mu_{A}(\id_{A}\odot\sigma) \\
    &=\chi\mu_{B}(\psi\odot\psi)(\id_{A}\odot\sigma)=\chi\mu_{B}(\psi\odot(\psi\sigma)) \\
    &=\chi\mu_{B}(\psi\odot\id_{B})
  \end{align*}
  A similar calculation yields $\chi\mu_{B}(\id_{B}\odot\phi)=\chi\mu_{B}(\id_{B}\odot\psi)$, so condition (2) indeed holds.
\end{proof}


\begin{prop}\label{prop:Mnd(D)hascoequalizers}
  Let $\dc{D}$ be a fibrant double category with stable local colimits. Then the category $\Mnd(\dc{D})$ has all
  coequalizers.
\end{prop}
\begin{proof}
  Throughout the proof, $U$ will denote the forgetful functor $\Mnd(\dc{D})\to\End(\dc{D})$ and $F\colon\End(\dc{D})\to\Mnd(\dc{D})$ will
  be its left adjoint, whose existence was established in \cref{prop:freemonadfunctorexistence1} under these assumptions. The counit of this adjunction will be denoted by
  $\epsilon$. Observe that, since the adjoint $F$ was established by an appeal to \cref{thm:totaladjointthm}, by an inspection of the proof we also have that the components of
  $\epsilon$ are globular morphisms.

  Consider a pair of monad morphisms in $\dc{D}$ as follows
  \begin{displaymath}
    \begin{tikzcd}
      X\ar[r,bul,"A"]\ar[d,"f"']\ar[dr,phantom,"\Two\phi"] & X\ar[d,"f"] \\
      Y\ar[r,bul,"B"'] & Y
    \end{tikzcd}
    \qquad
    \begin{tikzcd}
      X\ar[r,bul,"A"]\ar[d,"g"']\ar[dr,phantom,"\Two\psi"] & X\ar[d,"g"] \\
      Y\ar[r,bul,"B"'] & Y
    \end{tikzcd}
  \end{displaymath}
  Then consider the coequalizers $\gamma$ and $\delta$ in $\End(\dc{D})$ of the pairs $(U\phi,U\psi)$ and $(UFU\phi,UFU\psi)$ respectively
  \begin{displaymath}
    \begin{tikzcd}
      Y\ar[r,bul,"UB"]\ar[d,"q"']\ar[dr,phantom,"\Two\gamma"] & Y\ar[d,"q"] \\
      Q\ar[r,bul,"C"'] & Q
    \end{tikzcd}
    \qquad
    \begin{tikzcd}
      Y\ar[r,bul,"UFUB"]\ar[d,"q"']\ar[dr,phantom,"\Two\delta"] & Y\ar[d,"q"] \\
      Q\ar[r,bul,"D"'] & Q
    \end{tikzcd}
  \end{displaymath}
We will use the free image of those two coequalizers in order to construct the coequalizer $E$ of $\phi$ and $\psi$ in $\Mnd(\dc{D})$. 
In the process, we will make use of the above lemma that, however, refers only to \emph{globular} monad maps.

We should comment here on why we can assume that $\gamma$ and $\delta$ have the same vertical components. First, recall that having stable local colimits implies the existence of parallel colimits by \cref{rmk:term_conv}. Now, as in the proof of \cref{lem:(co)limits in End(D)}, not only does $\End(\dc{D})\to\dc{D}_0$ preserve colimits, namely the vertical component of both 2-cells must be a coequalizer for $(f,g)$ in $\dc{D}_0$, but also it can be thought of strictly doing so due to the opfibration structure of source/target, hence we can use the same $q$.  

We now obtain the following commutative diagram in $\Mnd(\dc{D})$. In more detail, the top two rows are coequalizers in $\Mnd(\dc{D})$ because $F$ is a left adjoint, and $\xi$ and $\zeta$ are induced by the
coequalizer property of $F\delta$ due to the commutativities involving the pairs
 $(\epsilon_{FUA},\epsilon_{FUB})$ and $(FU\epsilon_{A},FU\epsilon_{B})$ respectively.
Moreover, the first two columns lie in
  $\Mon(\dcD{X}{X},\odot,1_X)$ and $\Mon(\dcD{Y}{Y},\odot,1_Y)$ respectively, and are in fact the canonical presentations of $A$ and $B$ as
  Eilenberg-Moore algebras in these monadic categories over $\dcD{X}{X}$ and $\dcD{X}{X}$ respectively. In particular, these
  coequalizers are preserved by $U$.
  \begin{equation}\label{eq:3x3MndD}
    \begin{tikzcd}[sep=4em]
      FUFUA\ar[r,shift left=1ex,"FUFU\phi"]\ar[r,shift right=1ex,"FUFU\psi"']\ar[d,shift left=1ex,"\epsilon_{FUA}"]\ar[d,shift
      right=1ex,"FU\epsilon_{A}"'] & FUFUB\ar[d,shift left=1ex,"\epsilon_{FUB}"]\ar[d,shift right=1ex,"FU\epsilon_{B}"']\ar[r,two heads,"F\delta"] &
      FD\ar[d,shift left=1ex,"\xi"]\ar[d,shift right=1ex,"\zeta"']  \\
      FUA\ar[r,shift left=1ex,"FU\phi"]\ar[r,shift right=1ex,"FU\psi"']\ar[d,two heads,"\epsilon_{A}"'] & FUB\ar[d,two
      heads,"\epsilon_{B}"]\ar[r,two heads,"F\gamma"'] & FC \\
      A\ar[r,shift left=1ex,"\phi"]\ar[r,shift right=1ex,"\psi"'] & B &
    \end{tikzcd}
  \end{equation}

  We also observe that the morphisms $\xi,\zeta$ are globular. To see this, suppose that the vertical component of $\xi$ is $\xi_0\colon Q\to
  Q\in\dc{D}_0$. Since $\xi\circ F\delta=F\gamma\circ\epsilon_{FUB}$ and $\epsilon_{FUB}$ is itself globular, taking vertical components yields
  $\xi_{0}q=q$. But $q$ is a coequalizer in $\dc{D}_0$, so in particular an epimorphism, whence we deduce that $\xi_0=\id_{Q}$. A similar argument
  applies to $\zeta$, this time using that $FU\epsilon_{B}$ is globular.

  Applying the forgetful functor $U$ to the above diagram and considering the coequalizer $E$ of $U\zeta$ and $U\xi$, we have the following diagram in $\End(\dc{D})$, where the morphism $\theta$ is induced by 
  the coequalizer property of the second column.
  \begin{displaymath}
    \begin{tikzcd}[sep=4em]
      UFUFUA\ar[r,shift left=1ex,"UFUFU\phi"]\ar[r,shift right=1ex,"UFUFU\psi"']\ar[d,shift left=1ex,"U\epsilon_{FUA}"]\ar[d,shift
      right=1ex,"UFU\epsilon_{A}"'] & UFUFUB\ar[d,shift left=1ex,"U\epsilon_{FUB}"]\ar[d,shift right=1ex,"UFU\epsilon_{B}"']\ar[r,"UF\delta"] &
      UFD\ar[d,shift left=1ex,"U\xi"]\ar[d,shift right=1ex,"U\zeta"']  \\
      UFUA\ar[r,shift left=1ex,"UFU\phi"]\ar[r,shift right=1ex,"UFU\psi"']\ar[d,two heads,"U\epsilon_{A}"'] & UFUB\ar[d,two
      heads,"U\epsilon_{B}"]\ar[r,"UF\gamma"'] & UFC\ar[d,two heads,"e"] \\
      UA\ar[r,shift left=1ex,"U\phi"]\ar[r,shift right=1ex,"U\psi"'] & UB\ar[r,dashed,"\theta"'] & E
    \end{tikzcd}
  \end{displaymath}

We will now show that $E$ is the coequalizer of the original morphisms $\phi$ and $\psi$, namely that it obtains a monad structure. For that, we will apply \cref{lem:BettiCarboniStreetWalters} to the globular $\zeta$ and $\xi$, by showing that they have a common splitting in $\dcD{X}{X}$. First of all,  notice that the pairs of morphisms in the first two columns above both have globular common splittings in $\Mnd(\dc{D})$, given respectively by $F\eta_{UA}$ and $F\eta_{UB}$. Moreover, by naturality of $\eta$ we have
  \begin{align*}
    F\delta\circ F\eta_{UB}\circ FU\phi &=F(\delta\circ\eta_{UB}\circ U\phi)=F(\delta\circ UFU\phi\circ\eta_{UA})=F(\delta\circ UFU\psi\circ\eta_{UA})=F(\delta\circ\eta_{UB}\circ U\psi) \\
    &=F\delta\circ F\eta_{UB}\circ FU\psi
  \end{align*}
  By the coequalizer property of $FC$ this implies the existence of a unique $\sigma\colon FC\to FD\in\Mnd(\dc{D})$ such that $\sigma\circ
  F\gamma=F\delta\circ F\eta_{UB}$. Then we have that
  \begin{displaymath}
    \xi\circ\sigma\circ F\gamma=\xi\circ F\delta\circ F\eta_{UB}=F\gamma\circ\epsilon_{FUB}\circ F\eta_{UB}=F\gamma\circ\id_{FUB}=F\gamma
  \end{displaymath}
  \begin{displaymath}
    \zeta\circ\sigma\circ F\gamma=\zeta\circ F\delta\circ F\eta_{UB}=F\gamma\circ FU\epsilon_{B}\circ F\eta_{UB}=F\gamma\circ\id_{FUB}=F\gamma
  \end{displaymath}
  and because $F\gamma$ is an epimorphism in $\Mnd(\dc{D})$ we conclude that $\xi\sigma=\id_{FC}=\zeta\sigma$, hence $\sigma$ is a common splitting of $\zeta$ and $\xi$ also in $\dcD{X}{X}$. 

As a result, \cref{lem:BettiCarboniStreetWalters} applies and endows $E$ with a monad structure such that $e$ is the coequalizer of $\xi,\zeta$ in $\Mnd(\dc{D})$. Going back to \cref{eq:3x3MndD} in $\Mnd(\dc{D})$, which is now completed with $E$, we have that all three columns and the first two rows are coequalizers. Thus, it follows that the bottom row is also a coequalizer in $\Mnd(\dc{D})$.
\end{proof}

By well-known results of Linton \cite{Linton}, a category of Eilenberg-Moore algebras over a cocomplete base category is itself cocomplete as
soon as it has (reflexive) coequalizers. Hence, we immediately deduce the following.

\begin{cor}\label{cor:Mnd(D)cocomplete}
	Let $\dc{D}$ be a fibrant double category which has stable local colimits. Then the global category of monads
	$\Mnd(\dc{D})$ is cocomplete.
\end{cor}

\begin{proof}
	By \cref{lem:(co)limits in End(D)}, under these assumptions\footnote{Recall that due to our terminology convention of \cref{rmk:term_conv}, the existence of stable local colimits implies the existence of parallel colimits.} $\End(\dc{D})$ is cocomplete, and the monadic $\Mnd(\dc{D})$ has all coequalizers by  \cref{prop:Mnd(D)hascoequalizers}.
\end{proof}


\begin{ex}
	If $\ca{C}$ is a category with pullbacks and universal small colimits, then $\dc{S}\nc{pan}(\ca{C})$ has stable local colimits. So for such a $\ca{C}$, by \cref{thm:Mnd(D)monadic} and \cref{cor:Mnd(D)cocomplete} we have the well-known fact that the category of internal categories $\Cat(\ca{C})$ is cocomplete and monadic over the category of internal graphs $\nc{Grph}(\ca{C})$.
\end{ex}


\begin{ex}
	Consider a monoidal category $\ca{V}$ which is cocomplete and such that colimits are preserved by $\otimes$ in each variable. Then $\VMMat$ has stable local colimits, so the preceding two results again apply to yield that $\VCat$ is cocomplete and monadic over $\VGrph$. Once more, in the special case when $\ca{V}$ is symmetric monoidal closed, these recover results of H. Wolff \cite[Theorem 2.13, Corollary 2.14]{Wolff}.
\end{ex}

Before moving on to discussing local presentability of the category $\Mnd(\dc{D})$, we record also the following general observation concerning colimits in the category of monads. This gives a different type of assumptions on $\dc{D}$ that would ensure the existence of a particular type of colimits in $\Mnd(\dc{D})$ and should be compared to the discussion in \cref{rmk: on lp defi} (regarding preservation of filtered colimits).

\begin{prop}\label{MndColimits}
	If a fibrant double category $\dc{D}$ has parallel $\ca{I}$-colimits and the functors $-\odot-\colon\End(\dc{D})\to\End(\dc{D})$ and $1\colon\dc{D}_0\to\dc{D}_1$ preserve $\ca{I}$-colimits, then $\Mnd(\dc{D})$ has all $\ca{I}$-colimits
	and the inclusion $\Mnd(\dc{D})\to\End(\dc{D})$ creates them.
\end{prop}

\begin{proof}
	Suppose $D:\ca{I}\to \Mnd(\dc{D})$ is a diagram in $\Mnd(\dc{D})$. For any $i\in \ca{I}$ we get a monad
	\[
	(A_i:X_i\bular X_i,\mu_i,\eta_i)
	\]
	and for each arrow $\alpha:i\to j$ in $\ca{I}$, a morphism of monads
	\[
	\begin{tikzcd}
		X_i \arrow[r, bul, "A_i" {name=Di}] \arrow[d, "\mathfrak{s}(D\alpha)"']\ar[dr, phantom, "\Two D\alpha"] & X_i \arrow[d, "\mathfrak{t}(D\alpha)"] \\
		X_j \arrow[r, bul, "A_j"' {name=Dj}] & X_j
	\end{tikzcd}
	\]
	We will construct the colimit $(A:X\bular X,\mu,\eta)$ of $D$ in three steps.
	
	First consider $D$ as a diagram in $\End(\dc{D})$, i.e. take $\ca{I}\to\Mnd(\dc{D})\to \End(\dc{D})$.
	Since $\dc{D}$ is fibrant, the new diagram has a colimit by \cref{lem:(co)limits in End(D)}, say $(A:X\bular X, q_i:A_i\Rightarrow A)$. As in said proof, $\dc{D}$ having parallel colimits implies that $X$ actually \textit{is} the colimit of the diagram $\ca{I}\to\Mnd(\dc{D})\to \End(\dc{D})\to \dc{D}_0$ where $\End(\dc{D})\to\dc{D}_0$ is either $\mathfrak{s}$ or $\mathfrak{t}$.
	
	We can endow $A$ with a monad structure as follows. Since $-\odot -$ preserves $\ca{I}$-colimits, $(A\odot A, q_i\odot q_i)$ is a colimiting cocone of the diagram $\ca{I}\to \Mnd(\dc{D})\to \End(\dc{D})\xrightarrow{-\odot-} \End(\dc{D})$. But using each $A_i$'s multiplication, $(A:X\bular X,q_i\circ \mu_i)$ forms another cocone for this diagram since for any $\alpha:i\to j$ in $\ca{I}$ we get
	\begin{alignat*}{2}
		q_j\circ \mu_j\circ (D\alpha\odot Da)&= q_j\circ D_a \circ \mu_i\ \ \ &&\text{as $Da$ is a monad morphism}\\
		&=q_i\circ \mu_i\ \ \ &&\text{since $(A,q_i)$ is a cocone of $D$.}
	\end{alignat*}
	By universality of $A\odot A$, there is a unique $\mu: A\odot A\Rightarrow A$ such that $\mu (q_i\odot q_i)=q_i \mu_i$. We can verify that $\mu$ is associative using the associativity of the $\mu_i$'s and by definition of $\mu$. 
	
	Similarly, since $1:\dc{D}_0\to\dc{D}_1$ preserves $\ca{I}$-colimits, $(1_X:X\bular X,1_{q_i})$ is a colimiting cocone of the diagram $\ca{I}\to \Mnd(\dc{D})\to \End(\dc{D})\to \dc{D}_0\to \dc{D}_1$. But $(A:X\bular X, q_i \eta_i)$ is another cocone since
	\begin{alignat*}{2}
		q_i\circ \eta_j\circ 1_{\mathfrak{s}(D\alpha)}&= q_j\circ D\alpha\circ \eta_i\ \ \ &&\text{as $D\alpha$ is a monad morphism}\\
		&=q_i\circ \eta_i\ \ \ &&\text{since $(A,q_i)$ is a cocone of $D$.}
	\end{alignat*}
	By universality of $1_X$ there is a unique 2-cell $\eta:1_X\Rightarrow A$ so that $q_i \eta_i=\eta 1_{\mathfrak{s}(q_i)}$.
	As before, unitality is shown using its universal property and the unitality of the $\eta_i$'s, hence  $\eta$ is the required unit. 
	
	So $(A,\mu,\eta)$ is a monad, and furthermore the legs $q_i$'s of the colimiting cone in $\End(\dc{D})$ are morphisms of monads, by the above factorizations. Hence $(A,\mu,\eta)$ is a cocone in $\Mnd(\dc{D})$ of the initial diagram $D$. To show that $(A,\mu,\eta)$ is colimiting therein, take another cocone $((B:Y\bular Y,\mu_B,\eta_B),\phi_i)$, which constitutes a cocone in $\End(\dc{D})$ as well. Therefore there is some 2-cell $\phi:A\Rightarrow B$ in $\End(\dc{D})$. Since $-\odot -$ preserves colimits, $(A\odot A,q_i\odot q_i)$ is a colimiting cocone of $(-\odot -)D$ thus its legs $q_i\odot q_i$ are jointly epic. Therefore, from
	\[
	\mu_B(\phi\odot \phi)(q_i\odot q_i)= \mu_B(\phi q_i\odot \phi q_i) = \mu_B(\phi_i\odot \phi_i)=\phi_i\mu_i=\phi q_i\mu_i=\phi\mu (q_i\odot q_i)
	\]
	we get that $\mu_B(\phi\odot \phi)=\phi\mu$ which means that $\phi$ respects multiplication. With similar arguments, we get that $\phi$ also respects the units, thus $\phi$ is a monad map.
	
	The fact that the forgetful functor creates these colimits is evident from the above construction.
\end{proof}

\begin{rmk}
	To see that the assumptions in \cref{prop:Mnd(D)hascoequalizers} and \cref{MndColimits} are indeed of a different nature, consider $\dc{D}=\dc{S}\nc{pan}(\ca{C})$ for a category with pullbacks $\ca{C}$. In this case the condition that $\dc{D}$ have stable local $\ca{I}$-colimits translates to $\ca{C}$ having $\ca{I}$-colimits which are stable under pullback. On the other hand, $-\odot-\colon\nc{Grph}(\ca{C})\to\nc{Grph}(\ca{C})$ preserving $\ca{I}$-colimits would require the latter to \emph{commute} with pullbacks in $\ca{C}$, a property which in general is distinct from pullback-stability.
\end{rmk}

We close this section by establishing local presentability for the category of monads $\Mnd(\dc{D})$, for a locally presentable $\dc{D}$ (\cref{defi:lp}). Towards this end, we will need to impose extra conditions on $\dc{D}$ concerning the preservation of filtered colimits by the unit and horizontal composition functors. 
The following proof essentially generalizes the approach for the category of monoids on a locally presentable and monoidal category, see \cite[2.6]{MonComonBimon}.

\begin{thm}\label{thm:Mndlp}
  Let $\dc{D}$ be a locally presentable double category where $-\odot-\colon\End(\dc{D})\to\End(\dc{D})$ and $1\colon\dc{D}_0\to\dc{D}_1$ preserve filtered colimits. Then the category of monads $\Mnd(\dc{D})$ is locally presentable.
\end{thm}

\begin{proof}
  We define an endofunctor on the horizontal endo-1-cell category
 \begin{equation}\label{eq:defF1}
    F\colon\End(\dc{D})\xrightarrow{\phantom{AAAAAAAAAAAAA}}\End(\dc{D})\phantom{AAA}
  \end{equation}
  \begin{displaymath}
  \begin{tikzcd}
  X\ar[r,bul,"M"] & X
  \end{tikzcd}\quad\mapsto\quad\begin{tikzcd}(X\ar[r,bul,"M\odot M"] & X)\end{tikzcd}+\begin{tikzcd}(X\ar[r,bul,"1_X"] & X)\end{tikzcd}
  \end{displaymath}  
where ${}^f\alpha^f\mapsto(\alpha\circ\alpha)+1_f$. The category
  $\Alg(F)$ of $F$-algebras for this endofunctor has objects $A$ equipped with a morphism $FA\to A$ in $\End(\dc{D})$ (and no axioms are required to be satisfied). In other words, an $F$-algebra is a horizontal endo-1-cell $A\colon X\bular X$ equipped with two structure 2-cells
  \begin{displaymath}
    \begin{tikzcd}
      X\ar[drr,phantom,"\Two a_1"]\ar[d,"x_1"']\ar[r,bul,"A"] & X\ar[r,bul,"A"] & X\ar[d,"x_1"] \\
      X\ar[rr,bul,"A"'] && X
    \end{tikzcd}\quad
    \begin{tikzcd}
      X\ar[d,"x_2"']\ar[r,bul,"1_X"]\ar[dr,phantom,"\Two a_2"] & X\ar[d,"x_2"] \\
      X\ar[r,bul,"A"'] & X.
    \end{tikzcd}
  \end{displaymath}
  Morphisms are 2-cells ${}^f\alpha^f\colon A\Rightarrow B$ that commute with the above structure cells.
  It is immediate that this endofunctor is finitary, since both $-\odot -$ and $1_{(-)}$ are by assumptions. Moreover, since $\End(\dc{D})$ is locally presentable by \cref{thm:End(D) lp}, it follows that $\Alg(F)$ is also locally presentable, by standard facts about categories of algebras for an endofunctor, see e.g.~\cite[11.3.2]{InitialAlgebras}. 

 It is then the case that the category $\Mnd(\dc{D})$ of monads (\cref{Monadindoublecat}) can be viewed as a full subcategory of $\Alg(F)$: objects are $F$-algebras which on the one hand have globular structure 2-cells and on the other hand satisfy associativity and unitality axioms, and monad maps coincide with $F$-algebra maps. Hence, by showing that $\Mnd(\dc{D})$ is closed under limits and filtered colimits in $\Alg(F)$, the so-called ``reflection theorem'' \cite[2.48]{LocallyPresentable} will imply that $\Mnd(\dc{D})$ is itself a locally presentable category.
 

  We begin with limits. Consider a diagram $D\colon\ca{I}\to\Mnd(\dc{D})$ in $\Mnd(\dc{D})$
  and 
  take its limit 
  in $\Alg(F)$ 
  \begin{center}
    \begin{tikzcd}
      X\ar[d,"l_i"']\ar[r,bul,"L"]\ar[dr,phantom,"\Two \lambda_i"] & X\ar[d,"l_i"] \\
      X_i\ar[r,bul,"Di"'] & X_i
    \end{tikzcd}
  \end{center}
Since the locally presentable double category $\dc{D}$ has all parallel limits, and $\Alg(F)\to\End(\dc{D})\to\dc{D}_1$ creates all limits by standard endofunctor algebra facts and \cref{lem:(co)limits in End(D)},
the vertical components $(l_i\colon X\to X_i)_{\in\ca{I}}$ form a limit in $\dc{D}_0$ 
and in particular the $l_i$'s form a jointly monomorphic family. To show that $(L,\mu,\eta)$ is actually in $\Mnd(\dc{D})$, we must first show that $L$'s structure maps
  \begin{displaymath}
    \begin{tikzcd}
      X\ar[drr,phantom,"\Two \mu"]\ar[d,"x_1"']\ar[r,bul,"L"] & X\ar[r,bul,"L"] & X\ar[d,"x_1"] \\
      X\ar[rr,bul,"L"'] && X
    \end{tikzcd}\quad
    \begin{tikzcd}
      X\ar[d,"x_2"']\ar[r,bul,"1_X"]\ar[dr,phantom,"\Two \eta"] & X\ar[d,"x_2"] \\
      X\ar[r,bul,"L"'] & X
    \end{tikzcd}
  \end{displaymath}
  are globular. Since $\lambda_i$ are $F$-algebra maps, we have equalities
  \begin{displaymath}
    \begin{tikzcd}
      X\ar[drr,phantom,"\Two \mu"]\ar[d,"x_1"']\ar[r,bul,"L"] & X\ar[r,bul,"L"] & X\ar[d,"x_1"] \\
      X\ar[rr,bul,"L"']\ar[drr,phantom,"\Two\lambda_i"]\ar[d,"l_i"'] && X\ar[d,"l_i"] \\
      X_i\ar[rr,bul,"Di"'] && X_i
    \end{tikzcd}
    \quad
    =
    \quad
    \begin{tikzcd}
      X\ar[r,bul,"L"]\ar[d,"l_i"]\ar[dr,phantom,"\Two\lambda_i"] & X\ar[r,bul,"L"]\ar[d,"l_i"]\ar[dr,phantom,"\Two\lambda_i"] & X\ar[d,"l_i"] \\
      X_i\ar[d,equal]\ar[drr,phantom,"\Two\mu_i"]\ar[r,bul,"Di"'] & X_i\ar[r,bul,"Di"'] & X_i\ar[d,equal] \\
      X_i\ar[rr,bul,"Di"'] && X_i
    \end{tikzcd},\qquad
    \begin{tikzcd}
      X\ar[d,"x_2"']\ar[r,bul,"1_X"]\ar[dr,phantom,"\Two \eta"] & X\ar[d,"x_2"] \\
      X\ar[d,"l_i"']\ar[r,bul,"L"]\ar[dr,phantom,"\Two \lambda_i"] & X\ar[d,"l_i"] \\
      X_i\ar[r,bul,"Di"'] & X_i
    \end{tikzcd}
    \quad
    =
    \quad
    \begin{tikzcd}
      X\ar[d,"l_i"']\ar[r,bul,"1_X"]\ar[dr,phantom,"\Two 1_{l_i}"] & X\ar[d,"l_i"] \\
      X_i\ar[r,bul,"1_{X_i}"]\ar[dr,phantom,"\Two\eta_i"]\ar[d,equal] & X_i\ar[d,equal] \\
      X_i\ar[r,bul,"Di"'] & X_i
    \end{tikzcd}
  \end{displaymath}
that vertically imply that $l_{i}x_1=l_i$ and $l_{i}x_2=l_i$ for all $i\in\ca{I}$, thus $x_1=\id_X$ and $x_2=\id_X$.
To then show the associativity and unit axioms for $(L,\mu,\eta)$, it suffices to compose with the also jointly cancellable $\lambda_i$'s and use the corresponding axioms for each monad $(D_i,\mu_i,\eta_i)$ and that $\lambda_i$'s are algebra maps:
  \begin{align*}
    \lambda_i\mu(\mu\odot 1) &=\mu_i(\lambda_i\odot\lambda_i)(\mu\odot 1)=
    \mu_i(\mu_i\odot 1)(\lambda_i\odot\lambda_i\odot\lambda_i)=\mu_i(1\odot\mu_i)(\lambda_i\odot\lambda_i\odot\lambda_i) \\
    &
    =\mu_i(\lambda_i\odot\lambda_i)(1\odot\mu)=\lambda_i\mu(1\odot\mu)\\
    \lambda_i\mu(1\odot\eta) &=\mu_i(\lambda_i\odot\lambda_i)(1\odot\eta)=\lambda_i=\mu_i(\lambda_i\odot\lambda_i)(\eta\odot 1)=\lambda_i\mu(\eta\odot 1)
  \end{align*}
  where we suppressed horizontal composition associativity and unitality constraints.

  Now we turn to filtered colimits. Once again, consider a diagram $D\colon\ca{I}\to\Mnd(\dc{D})$
  with $\ca{I}$ filtered, and take its colimit in $\Alg(F)$
  \begin{equation}\label{gammai}
    \begin{tikzcd}
      X_i\ar[r,bul,"Di"]\ar[d,"q_i"']\ar[dr,phantom,"\Two\gamma_i"] & X_i\ar[d,"q_i"] \\
      Q\ar[r,bul,"C"'] & Q.
    \end{tikzcd}
  \end{equation}
Since $F$ preserves filtered colimits, again by standard endofunctor algebras facts we know that such colimits in $\Alg(F)$ are created by functor $\Alg(F)\to\End(\dc{D})$ (\cite{MonComonBimon}), 
and using \cref{lem:(co)limits in End(D)} and parallel colimits of $\dc{D}$, the vertical components $(q_i:X_i\to Q)$ again form a colimiting cocone in $\dc{D}_0$, and in particular form a jointly epimorphic family.

  To show that $(C,\mu,\eta)$ is in $\Mnd(\dc{D})$, one can use similar arguments as the ones above: first of all, the structure 2-cells of the colimit $(C,\mu,\eta)$ are globular like before.
 For the associativity and unit axioms for $(C,\mu,\eta)$, one can use a somewhat dual argument to the limit case, where in addition now we note that by assumptions, 
 composing the colimiting cocone \cref{gammai} horizontally with itself yields another colimit. In particular, the families of 2-cells 
 $\gamma_i\odot\gamma_i\odot\gamma_i$ are jointly epimorphic. So for example, associativity 
 can be deduced from the following calculation
  \begin{align*}
  \mu(\mu\odot 1)(\gamma_i\odot\gamma_i\odot\gamma_i) &=\mu(\gamma_i\odot\gamma_i)(\mu_i\odot 1)=\gamma_i\mu_i(\mu_i\odot 1)= \gamma_i\mu_i(1\odot\mu_i)=\mu(\gamma_i\odot\gamma_i)(1\odot\mu_i) \\
  &=\mu(1\odot\mu)(\gamma_i\odot\gamma_i\odot\gamma_i).
  \end{align*}
\end{proof}

\begin{rmk}
	We should note here that the assumption of preservation of filtered colimits by all functors $-\odot M$ and $M\odot-$, which is included in \cref{defi:lp} of a locally presentable double category, is not in fact needed in the above proof. Rather, for the purposes of this result, these conditions are substituted with the corresponding ones on $1\colon\dc{D}_0\to\dc{D}_1$ and $-\odot-\colon\End(\dc{D})\to\End(\dc{D})$, giving an example of what was discussed in \cref{rmk: on lp defi}. 
\end{rmk}

\begin{ex}
  The above assumptions hold in $\dc{D}=\VMMat$, for $\ca{V}$ a locally presentable monoidal category such that $\otimes$ preserves (coproducts and) filtered colimits in each variable. First of all, $\VMMat$ is then locally presentable by \cref{ex:lpdouble}. Moreover, the preservation of filtered colimits by $-\odot-\colon\VGrph\to\VGrph$ is essentially \cite[Lem.~3.2]{KellyLack}, while on can verify that $1\colon\Set\to\VMMat_1$ preserves all connected colimits. 
  
  We thus deduce that, under the stated assumptions on $\ca{V}$, the category of enriched categories $\VCat$ is locally presentable. 
In particular, this recovers \cite[Thm.~4.5]{KellyLack}.
\end{ex}

\begin{ex}
	Similarly, one can see that the result also applies to $\dc{D}=\dc{S}\nc{pan}(\ca{C})$, for $\ca{C}$ a locally presentable category. By \cref{ex:lpdouble}, $\dc{S}\nc{pan}(\ca{C})$ is locally presentable. The preservation of filtered colimits by the functor $1$ is immediate, while the corresponding property for $-\odot-\colon\nc{Grph}(\ca{C})\to\nc{Grph}(\ca{C})$ boils down to the commutativity of pullbacks and filtered colimits in $\ca{C}$. In this case, we deduce that $\ca{C}$ being locally presentable implies the same for $\Cat(\ca{C})$, the category of internal categories in $\ca{C}$, see e.g. \cite[Thm.~3.24]{Cosmoi}.
\end{ex}

\section{Comonads}

We now turn to the category $\Cmd(\dc{D})$ of comonads and comonad maps in double categories, recalled in \cref{Monadindoublecat}. It is known that even in the one-object case, namely comonoids in monoidal categories, the existence of cofree coalgebras is more demanding than its dual problem: for example, there does not exist a formula like the ``geometric series'' one for free monoids, used in the previous section. Inspired by Porst's approach \cite{MonComonBimon} concerning the existence of a right adjoint to $\Comon(\ca{V})\to\ca{V}$ for a monoidal category $\ca{V}$, we establish the existence of cofree comonads in double categories, dually to \cref{prop:freemonadfunctorexistence2}. 

\begin{prop}\label{prop:freecomonadexistence}
  Suppose that $\dc{D}$ is a locally presentable double category. Then the forgetful functor $U\colon\Cmd(\dc{D})\to\End(\dc{D})$ has a right adjoint.
\end{prop}

\begin{proof}
  The forgetful $U\colon\Cmd(\dc{D})\to\End(\dc{D})$ constitutes an opfibred 1-cell
  \begin{displaymath}
    \begin{tikzcd}
      \Cmd(\dc{D})\ar[rr,"U"]\ar[dr] && \End(\dc{D})\ar[dl] \\
      & \dc{D}_{0}
    \end{tikzcd}
  \end{displaymath}
  between the opfibrations $\Cmd(\dc{D})\to\dc{D}_0$ and $\End(\dc{D})\to\dc{D}_0$, see \cref{prop:MonComonfibred}, since cocartesian liftings are essentially the same.
Thus by the dual of \cref{thm:totaladjointthm}, $U$ has a right adjoint when for every object $X\in\dc{D}_0$, the restrictions
$U_{X}\colon\Cmd(\dc{D})_{X}\to\End(\dc{D})_{X}$ to the fibers have a right adjoint.

Since $\End(\dc{D})_{X}=\dcD{X}{X}$ and $\Cmd(\dc{D})_{X}=\Comon(\dcD{X}{X})$, this reduces to the cofree comonoid existence on the monoidal category $(\dcD{X}{X},\odot,1_X)$. Since $\dc{D}$ is locally presentable,
$\dcD{X}{X}$ is locally presentable by \cref{thm:End(D) lp}, 
  and its tensor product $\odot$ preserves filtered colimits in each
  variable by \cref{rmk:lpstable}. Therefore by \cite[\S~2.7]{MonComonBimon}, $U_X$ has a right adjoint. 
\end{proof}

As for the dual case of monads in double categories (\cref{thm:Mnd(D)monadic}), it is the case that comonads are comonadic over endomaps as soon as cofree comonads exist, hence we obtain the following result.

\begin{thm}
Suppose that $\dc{D}$ is a double category such that $U\colon\Cmd(\dc{D})\to\End(\dc{D})$ has a right adjoint. Then $\Cmd(\dc{D})$ is comonadic over $\End(\dc{D})$.
\end{thm}

In particular, combining the last two results yields the following.

\begin{cor}\label{cor:Cmd(D)comonadic}
	Let $\dc{D}$ be a locally presentable double category. Then $\Cmd(\dc{D})$ is comonadic over $\End(\dc{D})$.
\end{cor}

Finally, we establish conditions on $\dc{D}$ under which the category of comonads $\Cmd(\dc{D})$ is locally presentable, namely the dual of \cref{thm:Mndlp} but with a different proof strategy. This again generalizes the fact that the category of comonoids in a locally presentable monoidal category $\ca{V}$ is itself locally presentable \cite[\S~2.7]{MonComonBimon}. The current proof adopts a similar, but eventually more involved technique, expressing the category of comonads as an equifier of certain natural transformations between accessible functors and then a limit theorem.

\begin{thm}\label{thm:Cmdlp}
Let $\dc{D}$ be a locally presentable double category where $-\odot-\colon\End(\dc{D})\to\End(\dc{D})$ and $1\colon\dc{D}_0\to\dc{D}_1$ preserve filtered colimits, and furthermore suppose that $\dc{D}$ has stable local initial objects\footnote{As will become clear from the proof, only functors of the form $\wc{g}\odot\mi\odot\wh{f}$ need to preserve initial objects. We note that this does not follow, since due to fibrancy, such functors have a \emph{left} adjoint -- analogously to \cref{eq:adj}.}. Then the category $\Cmd(\dc{D})$ of comonads is locally presentable.
\end{thm}

\begin{proof}
  There is an endofunctor on the category of endomaps (which has products by \cref{lem:(co)limits in End(D)})
  \begin{equation}\label{eq:defF}
    F\colon\End(\dc{D})\xrightarrow{\phantom{AAAAAAAAAAAAAAAAAAAAA}}\End(\dc{D})\phantom{AAAAAAAA}
  \end{equation}
  \begin{displaymath}
    \bultwocell{X}{M}{X}{f}{Y}{N}{Y}{f}{\alpha}\quad\mapsto\quad
    \begin{tikzcd}
      (X\ar[r,bul,"M"]\ar[d,"f"']\ar[dr,phantom,"\Two\alpha"] &
        X\ar[r,bul,"M"]\ar[d,"f"']\ar[dr,bend right=8,phantom,"\Two\alpha"'] &
      X)\times(X\ar[r,bul,"1_X"]\ar[d,shift right=6,"f"]\ar[d,shift
      left=6,"f"']\ar[dr,bend left=8,phantom,"\Two1_f"] & X)\ar[d,"f"] \\
      (Y\ar[r,bul,"N"'] & Y\ar[r,bul,"N"'] & Y)\times
      (Y\ar[r,bul,"1_Y"'] & Y)
    \end{tikzcd}
  \end{displaymath}
 An object in the category of coalgebras $\Coalg(F)$ is $C\colon X\bular X$ together
  with a map $C\to FC$ in $\End(\dc{D})$ which breaks down to a pair of structure 2-cells
  \begin{equation}\label{eq:structureC}
    \begin{tikzcd}
      X\ar[rr,bul,"C"]\ar[drr,phantom,"\Two c_1"]\ar[d,"x_1"'] &&  X\ar[d,"x_1"] \\
      X\ar[r,bul,"C"'] & X\ar[r,bul,"C"'] & X
    \end{tikzcd}\;\;\mathrm{and}\;\;
    \bultwocell{X}{C}{X}{x_2}{X}{1_X}{X}{x_2}{c_2}
  \end{equation}
  and morphisms in  $\Coalg(F)$ respect these structure maps.
  This $F$-coalgebras category is locally presentable again by standard facts about endofunctor coalgebras, since $\End(\dc{D})$ is locally presentable by \cref{thm:End(D) lp} and $F$ is finitary: both $-\odot -$ and $1$ preserve filtered colimits by our assumptions and therefore so does $F$ as finite limits commute with filtered colimits in any locally presentable category.

Similarly to the proof of \cref{thm:Mndlp}, this category contains the category of comonads $\Cmd(\dc{D})$: on the one hand the structure 2-cells $c_1,c_2$ are not globular, and on the other hand there are no coassociativity nor counitality axioms present. 
To this end, we consider the full subcategory $\overline{\Coalg(F)}$ of $\Coalg(F)$ spanned by objects with globular structure maps, namely of the form
  \begin{equation}\label{eq:globstr}
    \begin{tikzcd}
      X\ar[rr,bul,"C'"]\ar[d,equal]\ar[drr,phantom,"\Two\delta"] && X\ar[d,equal] \\
      X\ar[r,bul,"C'"'] & X\ar[r,bul,"C'"'] & X
    \end{tikzcd}\qquad
    \globtwocell{X}{C'}{X}{X.}{1_X}{X}{\epsilon}
  \end{equation}
  This intermediate category is important, because the category of comonads in $\dc{D}$ shall be expressed as an equifier category of natural transformations (expressing the comonad axioms) between functors with domain $\overline{\Coalg(F)}$ which will in the end imply local presentability.
Thus, before describing the relevant transformations, we need to verify that $\overline{\Coalg(F)}$ is also locally presentable.

First of all, it can be verified that $\overline{\Coalg(F)}$ is accessibly embedded 
in $\Coalg(F)$ -- and is in fact closed under all colimits, thus cocomplete. We will now show that it is also closed under all subobjects and thus, by a remark under \cite[Cor.~2.36]{LocallyPresentable}, accessible and therefore locally presentable. This requires the following lemma,
whose proof we postpone so as to not interrupt the main argument, by also noting that here is where theorem's assumption regarding stable local initial objects is only used.
 
  \begin{lem}\label{lem:mono}
    Under the assumptions of \cref{thm:Cmdlp}, any monomorphism $\alpha_f\colon C_X\to D_Y$ in $\Coalg(F)$ has the property that its source/target vertical map $f\colon X\to Y$ is a monomorphism in $\dc{D}_0$.
  \end{lem}

In order to show the required closedness under subobjects, we now consider an object $C'\in\overline{\Coalg(F)}$ and a monomorphism $\alpha\colon A\to C'$ in $\Coalg(F)$, aiming to show that in fact $A\in\overline{\Coalg(F)}$. Since $\alpha$ is an $F$-coalgebra map, we have
that
  \begin{displaymath}
    \begin{tikzcd}
      Z\ar[rr,bul,"A"]\ar[d,"z_1"']\ar[drr,phantom,"\Two a_1"] && Z\ar[d,"z_1"] \\
      Z\ar[r,bul,"A"]\ar[d,"f"']\ar[dr,phantom,"\Two\alpha"] &
      Z\ar[r,bul,"A"]\ar[d,"f"']\ar[dr,phantom,"\Two\alpha"] & Z\ar[d,"f"] \\
      X\ar[r,bul,"C'"'] & X\ar[r,bul,"C'"'] & X
    \end{tikzcd}=
    \begin{tikzcd}
      Z\ar[rr,bul,"A"]\ar[d,"f"']\ar[drr,phantom,"\Two\alpha"] && Z\ar[d,"f"]\\
      X\ar[rr,bul,"C'"]\ar[d,equal]\ar[drr,phantom,"\Two \delta"] && X\ar[d,equal] \\
      X\ar[r,bul,"C'"'] & X\ar[r,bul,"C'"'] & X
    \end{tikzcd}\;\;\mathrm{and}\;\;
    \begin{tikzcd}
      Z\ar[r,bul,"A"]\ar[d,"f"']\ar[dr,phantom,"\Two\alpha"] & Z\ar[d,"f"] \\
      X\ar[r,bul,"C'"]\ar[d,equal]\ar[dr,phantom,"\Two \epsilon"] & X\ar[d,equal] \\
      X\ar[r,bul,"1_X"'] & X
    \end{tikzcd}=
    \begin{tikzcd}
      Z\ar[r,bul,"A"]\ar[d,"z_2"']\ar[dr,phantom,"\Two a_2"] & Z\ar[d,"z_2"] \\
      Z\ar[r,bul,"1_Z"]\ar[d,"f"']\ar[dr,phantom,"\Two 1_f"] & Z\ar[d,"f"] \\
      X\ar[r,bul,"1_X"'] & X
    \end{tikzcd}
  \end{displaymath}
and since $\alpha$ is a monomorphism in $\Coalg(F)$, the above lemma implies that $f$ is a monomorphism in $\dc{D}_0$. Hence from the left-hand side equation for the boundaries we get that $fz_1=f$ and so $z_1=\id_Z$, and from the right-hand side equation for the boundaries we get that $fz_2=f$ and so $z_2=\id_Z$. As a result, the $F$-coalgebra $A$ has globular structure maps, hence $\overline{\Coalg(F)}$ is closed under subobjects in $\Coalg(F)$.

  We can now express $\Cmd(\dc{D})$ as an equifier of natural transformations between finitary functors emanating from the locally presentable $\overline{\Coalg(F)}$ as follows: if $U$ is the forgetful functor, define
  \[
    \begin{tikzcd}[column sep=5pt,row sep=10pt]
      \overline{\Coalg(F)} \arrow[rr, "U", bend left=49] \arrow[rd, "U"', bend right] & {\Two \phi_1,\psi_1}                      & \End(\dc{D}) & \overline{\Coalg(F)} \arrow[rr, "U", bend left=49] \arrow[rd, "U"', bend right] & {\Two \phi_2, \psi_2}                            & \End(\dc{D}) & \overline{\Coalg(F)} \arrow[rr, "U", bend left=49] \arrow[rd, "U"', bend right] & {\Two \phi_3, \psi_3}                            & \End(\dc{D}) \\
      & \End(\dc{D}) \arrow[ru, "\odot^3"', bend right] &              &                                                                                 & \End(\dc{D}) \arrow[ru, "-\odot 1_{(-)}"', bend right] &              &                                                                                 & \End(\dc{D}) \arrow[ru, "1_{(-)}\odot -"', bend right] &
    \end{tikzcd}
  \]
  where 
  \[
    (\phi_1)_C=(\delta*\id_C)\circ\delta:C\to (C\odot C)\odot C\ \ \text{and}\ \ 
    (\psi_1)_C=(\id_C*\delta)\circ\delta:C\to C\odot (C\odot C)
  \]
  will be used to force coassociativity,  
  \[
    (\phi_2)_C=(\epsilon*\id_C)\circ\delta:C\to 1_{\mathfrak{s}(C)}\odot C\ \ \text{and}\ \ 
    (\psi_2)_C=\ell_C:C\to 1_{\mathfrak{s}(C)}\odot C
  \]
  for left counitality, and similarly
  \[
  (\phi_3)_C=(\id_C*\epsilon)\circ\delta:C\to C\odot 1_{\mathfrak{t}(C)}\ \ \text{and}\ \ (\psi_3)_C=r_C:C\to C\odot 1_{\mathfrak{t}(C)}
  \]
  for right counitality. The equifier of these natural transformations, namely the subcategory spanned by the objects $C$ such that $(\phi_i)_C=(\psi_i)_C$ for all $i=1,2,3$, is precisely $\Cmd(\dc{D})$, and all categories and functors involved are accessible by our assumptions and constructions. In particular, the forgetful functor preserves all (filtered) colimits because it is the composite $\overline{\Coalg(F)}\hookrightarrow\Coalg(F)\hookrightarrow\End(\dc{D})$ where both arrows in fact create all colimits -- the first one as seen above and the second one due as is the case for any endofunctor coalgebras category. 
  Hence by \cite[Lem.~2.76]{LocallyPresentable}, $\Cmd(\dc{D})$ is an accessible category.

  Finally, $\Cmd(\dc{D})$ is cocomplete by \cref{(co)limits in (co)monads} and is therefore locally presentable as required.
\end{proof}

We now return to the proof of the earlier lemma. Before showing the claim, let us discuss its meaning. Recall that although colimits are created by the forgetful $\Coalg(F)\to\End(\dc{D})$ for a
category of coalgebras for any endofunctor $F\colon\End(\dc{D})\to\End(\dc{D})$, limits in categories of coalgebras are in general not preserved by that forgetful functor. One case this
is true (e.g. dual of \cite[Prop.~4.6]{AdamekPorst}) is when $F$ preserves said limits; since $F$ in this case is given by \cref{eq:defF}, and horizontal composition 
is not expected to preserve e.g. pullbacks, monomorphisms are not in general preserved by $F$.
However, we can at least verify that the \emph{vertical} part of a monorphism in $\Coalg(F)$ is in fact a $\dc{D}_0$-monomorphism,
which turns out to be precisely what was needed in the above proof.

\begin{proof}[Proof of \cref{lem:mono}]
We will show that $\Coalg(F)\to\dc{D}_0$ preserves monomorphisms. In order to accomplish this, it will be convenient to factor this functor as a composition $\Coalg(F)\to\Coalg(P)\to\dc{D}_0$, where $P$ is the endofunctor
\[
P\colon\dc{D}_0\xrightarrow{\phantom{AA}}\dc{D}_0\phantom{aa}
\]
\[
\begin{tikzcd}[row sep=5pt, column sep=0pt]
	X \arrow[dd, "f"'] &         & X\times X \arrow[dd, "f\times f"] \\
	& \mapsto &                                   \\
	Y                  &         & Y\times Y
\end{tikzcd}
\]
and $\Coalg(F)\to\Coalg(P)$ is defined by mapping $(C,c_1,c_2)\mapsto (X,(x_1,x_2)\colon X\to X\times X)$, namely for an $F$-coalgebra $C$ as in \cref{eq:structureC} this functor discards the structure 2-maps and only keeps track of their vertical boundary maps, which can then be grouped as a single map into a product. Since $P$ clearly preserves pullbacks, the forgetful $\Coalg(P)\to\dc{D}_0$ creates pullbacks and hence in particular preserves monomorphisms. So it suffices to show that $\Coalg(F)\to\Coalg(P)$ also has this property. For the latter, we now show that this functor -- call it $U$ -- has a left adjoint $L\colon\Coalg(P)\to\Coalg(F)$.

Let us denote by $I_Z$ the initial object of $\dcD{Z}{Z}$ for any $Z\in\dc{D}_0$. The fundamental observation we will make use of is the following: if vertical morphisms $f\colon  Z\to X$ and $g\colon Z\to Y$ have been chosen, then there is a uniquely determined 2-cell
\begin{displaymath}
\begin{tikzcd}
	Z\ar[d,"f"']\ar[r,bul, "I_Z"]\ar[dr,phantom,"\Two"] & Z\ar[d,"g"] \\
	X\ar[r,bul,"M"'] & Y
\end{tikzcd}
\end{displaymath}
for any $M\colon X\bular Y$. The reason for this is that, because $I_Z$ is initial, there is a unique globular 2-cell 
\begin{displaymath}
\begin{tikzcd}
	Z\ar[d,equal]\ar[drrr,phantom,"\Two"]\ar[rrr,bul,"I_Z"] & & & Z\ar[d,equal] \\
	Z\ar[r,bul,"\wh{f}"'] & X\ar[r,bul,"M"'] & Y\ar[r,bul,"\wc{g}"'] & Z
\end{tikzcd}
\end{displaymath}
which then by fibrancy of $\dc{D}$ uniquely corresponds to a non-globular one of the previous form.

Now given a $P$-coalgebra $(X,x_1,x_2)$, we define $L(X,x_1,x_2)$ to be the $F$-coalgebra whose underlying 1-cell is $I_X$, with structure 2-cells 
\begin{displaymath}
\begin{tikzcd}
	X\ar[d,"x_1"']\ar[rr,bul,"I_X"]\ar[drr,phantom,"\Two"] && X\ar[d,"x_1"] \\
	X\ar[r,bul,"I_X"'] & X\ar[r,bul,"I_X"'] & X,
\end{tikzcd}
\qquad
\begin{tikzcd}
	X\ar[d,"x_2"']\ar[r,bul,"I_X"]\ar[dr,phantom,"\Two"] & X\ar[d,"x_2"] \\
	X\ar[r,bul,"1_X"'] & X
\end{tikzcd}
\end{displaymath}
the ones uniquely determined by their vertical components, as described above. Similarly, given a morphism $f\colon(X,x_1,x_2)\to(Y,y_1,y_2)$ of $P$-coalgebras, we define $Lf\colon L(X,x_1,x_2)\to L(Y,y_1,y_2)$ to be the uniquely determined 2-cell $I_X\Rightarrow I_Y$ with vertical components both equal to $f\colon X\to Y$. The fact that this is indeed an $F$-coalgebra morphism follows once more from the fact that morphisms beginning from $I_X$ are uniquely determined by their vertical components, where the desired equalities hold by the assumption that we have a $P$-coalgebra morphism to begin with. The functoriality also follows by uniqueness.

Finally consider an $F$-coalgebra morphism $L(Z,z_1,z_2)\Rightarrow (C,c_1,c_2)$ as below
\begin{displaymath}
\begin{tikzcd}
	Z\ar[d,"h"']\ar[r,bul,"I_Z"]\ar[dr,phantom,"\Two"] & Z\ar[d,"h"] \\
	X\ar[r,bul,"C"'] & X
\end{tikzcd}
\end{displaymath}
where $(C,c_1,c_2)$ is as in (\ref{eq:structureC}). Once again by the property of $I_Z$, the equalities making this 2-cell a morphism of $F$-coalgebras are valid iff they are so vertically, i.e. iff the pair of equalities $hz_1=x_1h$, $hz_2=x_2h$ hold. But the latter is precisely the statement that $h\colon Z\to X$ is a morphism of $P$-coalgebras. This shows that there is a bijection between $F$-coalgebra morphisms $L(Z,z_1,z_2)\Rightarrow (C,c_1,c_2)$ and $P$-coalgebra morphisms $(Z,z_1,z_2)\to U(C,c_1,c_2)$, which by similar arguments is easily seen to be natural.

\end{proof}

\begin{rmk}
Since $\VMMat$ is a locally presentable double category under the conditions of \cref{ex:lpdouble}, and in fact all colimits are preserved by functors $M\circ\mi\circ N$, we obtain that the category of cocategories $\VCocat$ (\cref{ex:cocat}) is a locally presentable category itself. In fact, the above arguments can be used to fix a gap in the proof of \cite[Prop.~4.31]{VCocats} in the specific context of matrices.
  
  Moreover, \cref{lem:mono} applied in $\VMMat$ says that any monomorphic $\ca{V}$-cofunctor comprises of a monomorphism between the set of objects of $\ca{V}$-cocategories, namely is injective
  on objects. As an idea, this appears as the dual of the well-known fact that epimorphic functors are surjective on objects (but not surjective on arrows).
\end{rmk}

\begin{rmk}\label{rmk:nonfibrantthm}
Using local presentability of comonads in a double category that satisfies the assumptions of \cref{thm:Cmdlp}, one could re-prove the first part of the main result \cite[Thm.~3.4.12]{SweedlerDouble} in a more direct -- yet slightly less informative\footnote{In loc.cit., we obtain also the carrier object of the enriched homs as extra information.} -- way. In more detail, if $\dc{D}$ is a braided monoidal closed and locally presentable double category, we obtain a (tensored and cotensored) enrichment of $\Mnd(\dc{D})$ in $\Cmd(\dc{D})$ as follows. We first form the commutative diagram
\begin{displaymath}
    \begin{tikzcd}[column sep=.6in]
      \Cmd(\dc{D})\ar[r,"{H(\mi,B)^\op}"]\ar[d] & \Mnd(\dc{D})^\op\ar[d] \\
      \End(\dc{D})\ar[r,"{H(\mi,B)^\op}"'] & \End(\dc{D})^\op
    \end{tikzcd}
  \end{displaymath}
 where $H\colon\Cmd(\dc{D})^\op\times\Mnd(\dc{D})\to\Mnd(\dc{D})$ is the ``internal-hom'' functor, coming from the monoidal closed structure of the double category (\cite[\S~3.2]{SweedlerDouble}) and restricts to the categories of monads and comonads exactly like convolution.

Then by \cref{cor:monadicity,cor:Cmd(D)comonadic}, both legs are comonadic by  therefore create colimits, and also the bottom functor preserves colimits since it has a right adjoint as $\End(\dc{D})$ is monoidal closed (\cite[Rem.~3.2.11]{SweedlerDouble}). Thus the top functor preserves colimits, and by a standard adjointness result (\cite[Thm.~5.33]{Kelly}), every cocontinuous functor from a locally presentable category has a right adjoint, here $H(\mi,B)^\op\dashv P(\mi,B)$. This implies the desired enrichment of monads in comonads, using a standard closed-action induced enrichment result originating back to \cite{enrthrvar}.
\end{rmk}

\bibliographystyle{alpha}
\bibliography{References}

\end{document}